\documentclass[11pt,reqno]{amsart}

\usepackage{amssymb} 
\usepackage{enumerate}
\usepackage{color}
\newtheorem{theorem}{Theorem}[section]

\newtheorem{prop}{Proposition}[section]
\newtheorem{lemma}[theorem]{Lemma}
\newtheorem{defn}{Definition}[section]
\theoremstyle{remark}
\newtheorem{remark}{Remark}[section]
\newtheorem{example}{Example}[section]
\numberwithin{equation}{section}

 \DeclareMathOperator\hdim{\dim_H}

\def\N{\mathbb{N}}
\def\R{\mathbb{R}}

 \usepackage{ulem}
\begin{document}

\title[]{Spectra of the Sierpi\'{n}ski type spectral measure and their Beurling dimensions}

\author[Jinjun Li]{Jinjun Li}

 \address[Jinjun Li]{School of Mathematics and Information Science, Guangzhou University, Guangzhou, 510006, P.~R.~China}
\email{li-jinjun@163.com}

\author{Zhiyi Wu}
\address[Zhiyi Wu]{School of Mathematics and Information Science, Guangzhou University, Guangzhou, 510006, P.~R.~China}
\email{zhiyiwu@126.com}



%\date{January 1, 2004}
%----------additions
%\dedicatory{To my boss}
%%% ----------------------------------------------------------------------

\begin{abstract}
In this paper, we study the structure of the spectra for the Sierpi\'{n}ski type spectral measure $\mu_{A,\mathcal{D}}$ on $\mathbb{R}^2$. We give a sufficient and necessary condition for the family of exponential functions
   $\{e^{-2\pi i\langle\lambda, x\rangle}: \lambda\in\Lambda\}$ to be a maximal orthogonal set in $L^2(\mu_{A,\mathcal{D}})$. Based on this result, we obtain a class of regular spectra of $\mu_{A,\mathcal{D}}$. Moreover, we discuss the Beurling dimensions of the spectra and obtain the optimal upper bound of Beurling dimensions of all spectra, which is in stark contrast with the case of self-similar spectral measure.  An intermediate property about the Beurling dimension of the spectra is obtained.
\end{abstract}

%%% ----------------------------------------------------------------------

%%% ----------------------------------------------------------------------
%\tableofcontents
\maketitle
	\tableofcontents
	\section{Introduction and results}
 One of the canonical problems of Harmonic analysis is to determine whether a given collection of functions is complete in a given Hilbert space. Among possible classes of functions, the family of exponential functions maybe the most important due to their good properties and universal applications.

We say that a Borel probability measure $\mu$ on $\mathbb{R}^d$ is a {\it spectral measure} if there exists a countable set $\Lambda\subseteq \mathbb{R}^d$ such that the family of exponential functions
 \begin{equation}\label{eq-basis}
 E(\Lambda)=\{e^{-2\pi i\langle\lambda, x\rangle}:\lambda\in\Lambda\}
\end{equation}
 forms an orthonormal basis for $L^2(\mu),$ the set of square integrable functions. In this case, the set $\Lambda$ is called a {\it spectrum} of $\mu$. For the convenience of expression, $\Lambda\subseteq \mathbb{R}^d$ is called  an {\it orthogonal set} (a {\it maximal orthogonal set}) of $\mu$ if the system of exponential functions defined by \eqref{eq-basis} is   an orthogonal set (a maximal orthogonal set)  of exponentials in the space $L^2(\mu)$.

In 1998, Jorgensen and Pederson \cite{JP98} showed that the standard middle-fourth Cantor measure $\mu_{4,\{0,2\}}$ is a spectral measure, which marks the entrance of Fourier analysis into the realm of fractals. Since then, much work has
been devoted to studying the spectrality of self-similar measures, self-affine measures and Moran measures, see \cite{AFL19,AH14,AW21,CLW,D12,De16,DHL14,DHLai19,DLau15,LDL,DJ07,JLL10,LMW,LW02,LZWC} and references therein. Around the same time, various new phenomena different from spectral theory for the Lebesgue measure have been found.  Among these, {\L}aba and Wang \cite{LW02} firstly discovered that there exists a self-similar spectral measure $\mu$ with a spectrum $\Lambda$ such that $2\Lambda$ is also a spectrum of $\mu$. This is rather surprising because the scaled set becomes more sparse but keeps the completeness. The exotic phenomenon naturally yields  the spectral eigenvalue problem \cite{FHW18,HTW19,D16,JKS11,LWu,WW20} and further spectral structure problem \cite{DHS09,DHL13,D16,HTW19,DDL19,AHH19}.

Dutkay, Han and Sun \cite{DHS09} firstly used the tool of maximal tree mapping (they called it maximal tree label) to give a characterization of  the maximal orthogonal set of exponentials in $L^2(\mu_{4,\{0,2\}})$. Moreover, they gave some sufficient conditions for a maximal orthogonal set to be a spectrum. Later, Dai, He and Lai   \cite{DHL13}  and Dai \cite{D16} improved the construction of the maximal tree mapping, which can separate the maximal orthogonal set of more general Cantor measure $\mu_{b,\{0,1,\ldots,q-1\}}$ with $q|b$ into two types: regular and irregular sets, where $\mu_{b,\{0,1,\ldots,q-1\}}$ is generated by an iterated function system of the form $\frac{1}{b}(x+i),i=0,1,\ldots,q-1$ and equally weighted probability. Furthermore, they obtained more general criteria for maximal orthogonal set of exponentials in $L^2(\mu_{b,\{0,1,\ldots,q-1\}})$ to be spectra. Recently, the spectral structure of more general self-similar measures, some self-affine measures and Moran measures have been studied \cite{DDL19,CDL22}.

Given a singular continuous spectral measure, it is a big challenge to completely characterize its spectra. To the best of our knowledge, there is no one singular continuous spectral measure whose spectra has been completely characterized. To overcome this issue, Duktay et al. \cite{DHSW11} tried to give a global characterization of all spectra for a spectral measure. They borrowed the notion of ``Beurling dimension", which was introduced to study the Gabor pseudo-frame for affine subspaces by Czaja, Kutyniok and Speegle \cite{CKS08}.

	Let $\Lambda\subseteq \mathbb{R}^d$ be a countable set. For $r>0,$ the \textit{upper $r$-Beurling density} of $\Lambda$ is defined by
\[
D_r^+(\Lambda)=\limsup\limits_{h\to \infty}\sup_{x\in \mathbb{R}^d}\frac{\#(\Lambda \cap B(x,h))}{h^r},
\]
where  $B(x,h)$ is the open ball centered at $x$ with radius $h$.
It is not difficult to prove that there is a critical value $r$ at which $D_r^+(\Lambda)$ jumps from $\infty$ to $0$. This critical value is defined as the \textit{upper Beurling dimension} of $\Lambda$, i.e.,
\[
\dim_{Be} (\Lambda)=\inf\{r:D_r^+(\Lambda)=0\}=\sup\{r:D_r^+(\Lambda)=\infty\}.
\]

 Duktay et al. \cite {DHSW11} proved that the Beurling dimension of the spectra of self-similar spectral measure which satisfies the open set condition is not greater than the Hausdorff dimension of its support, and they may equal under some mild condition. Their results were subsequently extended to more general self-similar spectral measure and some affine measure \cite{HKTW18,TW21,ZX20}. On the other hand, there exist arbitrarily sparse spectra for many singular spectral measures \cite{AL20,DHL13}, i.e., there exist spectra with Beurling dimensions zero. Recently, the authors found that the (quasi) Beurling dimensions of the spectra of some Moran self-similar spectral measures (including the spectral Bernoulli convolution)  possess an intermediate value property, i.e., the (quasi) Beurling dimensions of their spectra can attain any prescribed value from zero to the upper entropy dimension of the corresponding measure \cite{LWu122,LWu222}. These results are in stark contrast with the case of the Lebesgue measure.

In this paper, we study the spectral structure and the intermediate value property of the Beurling dimensions of spectra for a class of self-affine spectral measure $\mu=\mu_{A,\mathcal{D}}$, which satisfies
  \begin{equation}\label{eqiis}
  \mu_{A,\mathcal{D}}(\cdot)=\frac{1}{\#\mathcal{D}}\sum_{d\in\mathcal{D}}\mu_{A,\mathcal{D}}(A(\cdot)-d),
  \end{equation}
  where $A=\begin{pmatrix}
		n & 0 \\
		0 & m \\
	\end{pmatrix}$ with $1<n\le m$ and $n,m\in\mathbb{R}$ is an expanding matrix and
\[
\mathcal{D}=\left\{
\begin{pmatrix}
		0  \\
		0 \\
	\end{pmatrix}, \begin{pmatrix}
		1  \\
		0 \\
	\end{pmatrix},\begin{pmatrix}
		0  \\
		1  \\
	\end{pmatrix}\right\}.
\]
We call $\mu_{A,\mathcal{D}}$ the {\it Sierpi\'{n}ski-type measure}, which plays an important role in fractal geometry and geometric measure theory.

 In \cite{DJ07}, Dutkay and Jorgensen considered the case that $n=m\in\mathbb{Z}$ and proved that $\mu_{A,\mathcal{D}}$ is a spectral measure if and only if $n\in 3\mathbb{Z}$. Subsequently, Li \cite{JLL10} extended Dutkay and Jorgense's result to the case $n\neq m$ and proved that $\mu_{A,\mathcal{D}}$ is a spectral measure if and only if $n,m\in 3\mathbb{Z}$. In \cite{DLau15}, Deng and Lau made a start in studying the spectrality of $\mu_{A,\mathcal{D}}$ when $A$ is a real expanding matrix with $n=m\in\mathbb{R}$ and they proved the only spectral Sierpi\'{n}ski measure are of $n\in 3\mathbb{Z}$. Recently, Dai, Fu and Yan \cite{DFY21} completely characterized the spectrality of $\mu_{A,\mathcal{D}}$ when $A$ is a real expanding matrix. They proved that $\mu_{A,\mathcal{D}}$ is a spectral measure if and only if $n,m\in 3\mathbb{Z}$. So, throughout this paper, we always make the following assumption:

 $\mathbf{Assumption}$. Assume that $n=3q_1,m=3q_2$ with $q_1,q_2\in\mathbb{Z}$ and $1\leq q_1\leq q_2$.

Our first result establishes the following characterization of maximal
orthogonal set of the measure $\mu_{A,\mathcal{D}}$ via  maximal tree mapping (see Definition \ref{deftree}). It extends the result of An, Dong and He  \cite{ADH22}, which is about the Sierpi\'{n}ski type self-similar spectral measure (i.e., $n=m$), to self-affine case (i.e., $n\not=m$). The original idea is due to Dutkay et al. \cite{DHS09} who characterized the maximal orthogonal set of $\mu_{4,\{0,2\}}$ by using the so-called tree labeling method.

\begin{theorem}\label{max}
Let $\Lambda$ be a countable subset of $\mathbb{R}^2$ containing $\textbf{0}$. Then $\Lambda$ is a maximal orthogonal set of $\mu_{A,\mathcal{D}}$ if and only if there exists a maximal tree mapping $\tau$ such that $\Lambda=\tau^*(\Theta_3^\tau)$.
\end{theorem}

When $\tau$ is a regular mapping (see Definition \ref{remal}),  on the basis of Theorem \ref{max}, we have the following:

\begin{theorem}\label{ppd}
Let $\tau$ be a regular mapping and $\Lambda(\tau)=\tau^*(\Theta_3^\tau)$. If
\[
\ell_{\max}:=\max_{k\geq 1}\{\ell_k\}<\infty,
\]
then $\Lambda(\tau)$ is a spectrum of $\mu_{A,\mathcal{D}}$.
\end{theorem}

We will introduce the above notations and notion precisely in Section \ref{sec23}.
\begin{remark}
In fact, Theorem \ref{ppd} can be extended to more general maximal tree mapping. However, Theorem \ref{ppd} is enough for us to consider the intermediate property of Beurling dimensions with respect to the spectra of $\mu_{A,\mathcal{D}}$, which we are interested in.
\end{remark}
 Unlike the self-similar spectral measures, we don't know the exact upper bound of the Beurling dimensions of spectra for $\mu_{A,\mathcal{D}}$.  Tang and the second author \cite{TW21}, Zhang and Xiao \cite{ZX20} independently proved that for any orthogonal set $\Lambda$ of $\mu_{A,\mathcal{D}}$,
$\dim_{Be}(\Lambda)\leq\frac{\log3}{\log 3q_1}$. Moreover,  if $\Lambda$ is a spectrum of $\mu_{A,\mathcal{D}}$, under the condition that
\begin{eqnarray}\label{eq-upperconditioin1}
\sup_{\lambda\in\Lambda}\inf_{\gamma\in\Lambda}\|A^{-p}\lambda-\gamma\|<+\infty,
\end{eqnarray}
 we have
\[
\frac{\log3}{\log 3q_2}\leq \dim_{Be}(\Lambda)\leq \frac{\log 3}{\log 3q_1}.
\] In \cite{LWu322}, the Beurling dimension of a class of canonical spectra of $\mu_{A,\mathcal{D}}$, which is believed to have the largest Beurling dimension in all spectra, is determined and it equals $\frac{\log 3}{\log 3q_2}$. We will show that the value $\frac{\log 3}{\log 3q_2}$ is in fact the exact upper of Beurling dimensions of all spectra (Theorem \ref{thexu}). It is strictly less than the Hausdorff dimension of the support of $\mu_{A,\mathcal{D}}$  ($=\frac{\log\left(2^u+1\right)}{\log 3q_1},$ where $u=\log 3q_1/\log 3q_2$) , since $\frac{\log 3}{\log 3q_2}<\frac{\log\left(2^u+1\right)}{\log 3q_1}$ by an easy computation. This is in stark contrast with the self-similar case.

\begin{theorem}\label{thexu}
For any  orthogonal set $\Lambda$ of $\mu_{A,\mathcal{D}}$, we have that
$\dim_{Be}(\Lambda)\leq\frac{\log3}{\log 3q_2}$. Moreover, suppose $\Lambda$ is a spectrum of  $\mu_{A,\mathcal{D}}$, and if there exists an integer $p\geq1$ such that \eqref{eq-upperconditioin1} holds,
then $\dim_{Be}(\Lambda)=\frac{\log 3}{\log 3q_2}$.
\end{theorem}
\begin{remark}
Theorem 1.2 in \cite{Shi} and Theorem 1.2 in \cite{F19} tell us that for any orthogonal set $\Lambda$ of $\mu_{A,\mathcal{D}}$, we have that $\dim_{Be}(\Lambda)\leq \dim_e(\mu_{A,\mathcal{D}})$, where $\dim_e (\mu_{A,\mathcal{D}})$ denotes the entropy dimension of the measure $\mu_{A,\mathcal{D}}$; see Section \ref{rege1} for the definition. A simple computation shows that $\frac{\log 3}{\log 3q_2}<\dim_e(\mu_{A,\mathcal{D}})<\frac{\log 3}{\log 3q_1}$ (see Proposition \ref{propeed}). So, Theorem \ref{thexu} implies that the entropy dimension is not an optimal upper bound for $\mu_{A,\mathcal{D}}$, which is in stark contrast with the self-similar case \cite{DHSW11,HKTW18}.
\end{remark}

On the other hand, Theorem 1.3 in \cite{AL20} tells us that there exists a spectrum of $\mu_{A,\mathcal{D}}$ whose Beurling dimension is zero. So, there is a natural question: does there exist some spectrum of $\mu_{A,\mathcal{D}}$ with other Beurling dimension except $0$ and $\frac{\log3}{\log 3q_2}$?
The final contribution of this paper is the following result on the intermediate property of the Beurling dimensions of spectra of $\mu_{A,\mathcal{D}}$, which answers Problem 3.3 in \cite{LWu322}.

\begin{theorem}\label{thexu1}
For any $t\in\left[0,\frac{\log3}{\log 3q_2}\right]$, there exist uncountably many spectra $\Lambda_t$ of $\mu_{A,\mathcal{D}}$ such that their Beurling dimensions are equal to $t$.
\end{theorem}

The paper is structured as follows. We prove Theorem \ref{max} in Section \ref{sec23}. We
prove Theorem \ref{ppd} in Section \ref{rege}. We prove Theorem \ref{thexu} in Section \ref{rege1}.  Finally,  We prove Theorem \ref{thexu1} in Section \ref{rege1}.

\section{Maximal orthogonal sets of $\mu_{A,\mathcal{D}}$}\label{sec23}

In \cite{DHS09}, Dutkay et al. gave a complete characterization of all the maximal orthogonal sets of $\mu_{4,\{0,2\}}$ by using the maximal tree labelling method and obtained a sufficient condition for a maximal orthogonal set to be a spectrum.

There are many papers studying the maximal orthogonal sets and spectral structure of different kind of fractal spectral measures by extending the method of tree labelling method, see \cite{ADH22,DHL13,CDL22,DW19,WY}. In particular, An, Dong and He \cite{ADH22} recently obtained a complete characterization of the maximal orthogonal sets of Sierpi\'{n}ski-type self-similar spectral measure and gave a sufficient condition for a maximal orthogonal set to be a spectrum. In this paper,  we continue the line and consider the Sierpi\'{n}ski-type self-affine spectral measures, which are usually far more subtle than self-similar spectral measures.

For a Borel measure $\mu$, its  orthogonal set is closely related to the zero set $\mathcal{Z}(\hat{\mu})$ of $\mu$, which is defined by $\mathcal{Z}(\hat{\mu}):=\{\xi\in\mathbb{R}:\hat{\mu}(\xi)=0\}$, where
\[
\hat{\mu}(\xi)=\int e^{-2\pi i\langle \xi,x\rangle}\mathrm{d}\mu(x)
\]
is the Fourier transform of $\mu.$  Precisely, it is easy to see  that for any orthogonal set of $\mu$,
\[
(\Lambda-\Lambda)\setminus \{\mathbf{0}\}\subseteq\mathcal{Z}(\hat{\mu}).
\] From \eqref{eqiis}, we have
\[
\hat{\mu}_{A,\mathcal{D}}(\xi)=\prod_{j=1}^{\infty}m_{\mathcal{D}}(A^{-j}\xi),~\xi\in\mathbb{R}^2,
\]
where
\[
m_{\mathcal{D}}(x)=\frac{1}{\#\mathcal{D}}\sum_{d\in\mathcal{D}}e^{-2\pi i\langle d,x\rangle}.
\]
It is easy to see that the zero set of $m_{\mathcal{D}}$ is
\[
\mathcal{Z}(m_{\mathcal{D}})=\{\xi:m_{\mathcal{D}}(\xi)=0\}=\left(\frac{1}{3}\begin{pmatrix}
		1 \\
		2\\
	\end{pmatrix}+\mathbb{Z}^2\right)\cup\left(\frac{2}{3}\begin{pmatrix}
		1 \\
		2\\
	\end{pmatrix}+\mathbb{Z}^2\right)
\]
and therefore
\begin{equation}\label{eqmle33}
\begin{split}
\mathcal{Z}(\hat{\mu}_{A,\mathcal{D}})&=\{\xi:\hat{\mu}_{A,\mathcal{D}}(\xi)=0\}=\bigcup_{k=1}^{\infty}A^k\mathcal{Z}(m_{\mathcal{D}})\\
&=\bigcup_{k=1}^{\infty}A^k\left(\frac{1}{3}\begin{pmatrix}
		1 \\
		2\\
	\end{pmatrix}+\mathbb{Z}^2\right)\cup\left(\frac{2}{3}\begin{pmatrix}
		1 \\
		2\\
	\end{pmatrix}+\mathbb{Z}^2\right)\\
&=\bigcup_{k=1}^{\infty}A^{k-1}\left(\begin{pmatrix}
		q_1 \\
		2q_2\\
	\end{pmatrix}+A\mathbb{Z}^2\right)\cup\left(\begin{pmatrix}
		2q_1 \\
		4q_2\\
	\end{pmatrix}+A\mathbb{Z}^2\right).
\end{split}
\end{equation}

Before giving the notion of maximal tree mapping, we need some preparations. Recall that $n=3q_1, m=3q_2$. Define
\[\begin{split}
\Gamma_{q_1,q_2}&=\begin{pmatrix}
		3q_1 & 0 \\
		0 & 3q_2 \\
	\end{pmatrix}\left[-\frac{1}{2},\frac{1}{2}\right)^2\cap \mathbb{Z}^2\\&=\left\{\begin{pmatrix}
		k \\
		l \\
	\end{pmatrix}\in\mathbb{Z}^2:-\frac{3q_1}{2}\leq k<\frac{3q_1}{2},-\frac{3q_2}{2}\leq l<\frac{3q_2}{2}\right\}.
\end{split}\]
We immediately have the following simple fact:
\begin{lemma}
$\Gamma_{q_1,q_2}$ is a complete residual system $\pmod{A}$ in $\mathbb{Z}^2$, i.e.,
\[
\mathbb{Z}^2=\bigcup_{r\in\Gamma_{q_1,q_2}}(A\mathbb{Z}^2+r),
\]
and the union is disjoint.
\end{lemma}

For $x\in \mathbb{R}$, we let $[x]$ denote the greatest integer less than or equal to $x$. The following lemma is well-known. We  present the proof here for completeness.

\begin{lemma}\label{exg}
Let $b\ge 2$ be an integer. For any $k\in \mathbb{Z},$  there exist a unique sequence $\{d_n\}$ with $d_n\in \{-[\frac{b}{2}],-[\frac{b}{2}]+1,\ldots,b-1-[\frac{b}{2}]\}(=[-\frac{b}{2},\frac{b}{2})\cap\mathbb{Z})$ for $n\ge 1$ and $d_n=0$ for sufficiently large $n$ such that
\[
k=\sum_{n=1}^\infty b^{n-1}d_n.
\]
\end{lemma}
\begin{proof}
Without loss of generality, we assume $b$ is even since the proof of the case that $b$ is odd is similar. Let $k\in \mathbb{Z}$.  If $|k|\le b$,

we consider the following three cases. \textsc{Case 1}. if $-\frac{b}{2}\leq k<\frac{b}{2}$, then let $d_1=k$ and $d_n=0$ for $n\ge 2$. \textsc{Case 2}. if $\frac{b}{2}\leq k\leq b$, then let $d_1=k-b, d_2=1$ and $d_n=0$ for $n\ge 3$. \textsc{Case 3}. if $-b\leq k<-\frac{b}{2}$, then let $d_1=b+k, d_2=-1$ and $d_n=0$ for $n\ge 3$.

If $|k|>b$, then there exist a unique positive integer $n_0$ and  two integer numbers $\ell, s\in(-b,b)\cap\mathbb{Z}$ such that $k=b^{n_{0}}\ell+s$. Then we decompose $\ell$ and $s$ by the procedure in the proceeding paragraph and obtain the desired result.
\end{proof}

The following result follows immediately from Lemma \ref{exg}.
\begin{prop}\label{exp}
For each $\omega=(\omega_1,\omega_2)\in \mathbb{Z}^2$, there exists a unique expansion of $\omega$ with respect to $A$, that is,
\[
\omega=\sum_{k=1}^\infty A^{k-1}c_k,
\]
where $c_k\in \Gamma_{q_1,q_2}$ for $k\ge 1$, and $c_k=\begin{pmatrix}
		0 \\
		0 \\
\end{pmatrix}$ for sufficiently large $k.$
\end{prop}

We can decompose $\Gamma_{q_1,q_2}$ into the following disjoint union.
\begin{lemma}
Write \[
\mathcal{C}_{q_1,q_2}=\left\{\begin{pmatrix}
		 0 \\
		0  \\
	\end{pmatrix},\begin{pmatrix}
		 q_1 \\
		-q_2  \\
	\end{pmatrix},\begin{pmatrix}
		 -q_1 \\
		q_2  \\
\end{pmatrix}\right\}.
\]
 Then,
\begin{equation}\label{decom}
\Gamma_{q_1,q_2}=\bigcup_{a\in\mathcal{E}_{q_1}}(a+\mathcal{C}_{q_1,q_2})\pmod{A}
\end{equation}
where
\[
\mathcal{E}_{q_1}=\left\{
\begin{pmatrix}
		a_1 \\
		a_2 \\
\end{pmatrix}\in\Gamma_{q_1,q_2}:-\frac{q_1}{2}\leq a_1<\frac{q_1}{2} \right\}.
\]
Moreover, the union in \eqref{decom} is disjoint.
\end{lemma}

\begin{proof} Observe that
 \[
\bigcup_{a\in\mathcal{E}_{q_1}}(a+\mathcal{C}_{q_1,q_2})\pmod{A}=\bigcup_{b\in\mathcal{C}_{q_1,q_2}}(b+\mathcal{E}_{q_1})\pmod{A}
\]
with $(b+\mathcal{E}_{q_1})\pmod{A}\subseteq\Gamma_{q_1,q_2}$ and  the union of right hand of the above equality is disjoint. For any $b\in\mathcal{C}_{q_1,q_2}$, we claim that  $b+\mathcal{E}_{q_1}\pmod{A}$ contains the same number of elements as $\mathcal{E}_{q_1}$. In fact, if $b=\begin{pmatrix}
		 0 \\
		0\\
	\end{pmatrix}$, it is trivial. We prove it for $b=\begin{pmatrix}
		 q_1 \\
		-q_2  \\
	\end{pmatrix}$ and the other case is similar. For any two distinct elements $\begin{pmatrix}
		 a_1 \\
		a_2  \\
	\end{pmatrix}$ and $\begin{pmatrix}
		 a_1'\\
		a_2'  \\
	\end{pmatrix}$ in $\mathcal{E}_{q_1}$, if $\begin{pmatrix}
		 a_1 \\
		a_2  \\
	\end{pmatrix}+\begin{pmatrix}
		 q_1 \\
		-q_2  \\
	\end{pmatrix}\pmod{A}=\begin{pmatrix}
		 a_1' \\
		a_2'  \\
	\end{pmatrix}+\begin{pmatrix}
		 q_1 \\
		-q_2  \\
	\end{pmatrix}\pmod{A}$, then $\begin{pmatrix}
		 a_1-a_1' \\
		a_2-a_2' \\
	\end{pmatrix}=\begin{pmatrix}
		 0 \\
		0  \\
	\end{pmatrix}\pmod{A}$. If $a_1\neq a_1'$, this is a contradiction since $a_1-a_1'\in(-q_1,q_1)\setminus\{0\}$; if $a_1=a_1'$, this is also a contradiction since $a_2-a_2'\in(-3q_2,3q_2)\setminus\{0\}$. Then we obtain the above claim. On the other hand, $(b+\mathcal{E}_{q_1})\pmod{A}$ are all integer vectors for any $b$, which implies that
\[
\Gamma_{q_1,q_2}=\bigcup_{b\in\mathcal{C}_{q_1,q_2}}(b+\mathcal{E}_{q_1})\pmod{A}
\]
and thus $\Gamma_{q_1,q_2}=\bigcup_{a\in\mathcal{E}_{q_1}}(a+\mathcal{C}_{q_1,q_2})\pmod{A}$. If there exist $a=\begin{pmatrix}
		 a_1 \\
		a_2  \\
	\end{pmatrix}$ and $a'=\begin{pmatrix}
		 a_1' \\
		a_2'  \\
	\end{pmatrix}$ in $\mathcal{E}_{q_1}$ such that $(a+\mathcal{C}_{q_1,q_2})\pmod{A}\cap (a'+\mathcal{C}_{q_1,q_2})\pmod{A}\neq\emptyset$, we can obtain similar contradiction as the above argument.
\end{proof}
\begin{remark}
Similarly, we can define \[
\mathcal{E}_{q_2}=\left\{\begin{pmatrix}
		a_1 \\
		a_2 \\
	\end{pmatrix}\in\Gamma_{q_1,q_2}:-\frac{q_2}{2}\leq a_2<\frac{q_2}{2} \right\}
\]
and then $\Gamma_{q_1,q_2}$ can be decomposed into the  following disjoint classes:
\begin{equation}\label{eqc2}
\Gamma_{q_1,q_2}=\bigcup_{a\in\mathcal{E}_{q_2}}(a+\mathcal{C}_{q_1,q_2})\pmod{A}.
\end{equation}
\end{remark}

Write $\Theta_3=\{-1,0,1\}$. For $n\ge 1$, denote all the words  with length $n$  by $\Theta_3^n=\{I=i_1i_2\cdots i_n: \mbox{all} ~ i_k\in\Theta_3\}$, all the finite words by $\Theta_3^*=\bigcup_{n=1}^\infty \Theta_3^n$ and all the infinite words by $\Theta_3^{\infty}=\{I=i_1i_2\cdots: \mbox{all} ~ i_k\in\Theta_3\}$. For any $I=i_1i_2\cdots i_k\in\Theta_{3}^{k},J=j_1j_2\cdots\in\Theta_{3}^{\infty}\cup\Theta_{3}^{\ast}$, let $IJ=i_1i_2\cdots i_kj_1j_2\cdots$. For $I\in\Theta_{3}^{\ast}$, we denote the infinite  word $II\cdots$ by $I^\infty$.

Now, we  give the definition of maximal tree mapping.

\begin{defn}\label{deftree}
  We say that a mapping $\tau$ from $\Theta_3^*$ to $\Gamma_{q_1,q_2}$ is a {\it maximal tree mapping} if
  \begin{enumerate}[(i)]
\item $\tau(0^ki_{k+1})=i_{k+1}\begin{pmatrix}
		q_1 \\
		-q_2 \\
	\end{pmatrix}$ for  $k\ge 0$ and $i_{k+1}\in\Theta_3$; here we define $0^0i_1=\emptyset i_1=i_1$;
  \item  For any $k\ge 1$ and $Ij\in\Theta_3^{k+1}$ with $I\ne 0^k, j\in\Theta_3$, $\tau(Ij)\equiv e_I+j\begin{pmatrix}
		q_1 \\
		-q_2 \\
	\end{pmatrix}\pmod{A}$ for some $e_I\in\mathcal{E}_{q_1}$;
  \item For  $I\in\Theta_3^*$, there exists $J\in\Theta_3^*$ such that $\tau(IJ0^n)=\textbf{0} $ for all sufficiently large $n$ .

  \end{enumerate}
    \end{defn}
\begin{remark}
If we choose the decomposition \eqref{eqc2} of $\Gamma_{q_1,q_2}$, then we require the $e_I$ in (ii) of Definition \ref{deftree} belongs to $\mathcal{E}_{q_2}$.
\end{remark}

Given a maximal tree mapping $\tau$, define
 \[\Theta_3^\tau=\{I\in\Theta_3^\infty: \text{ $\tau(I|_n)=\textbf{0}$ for sufficiently large $n$}\}.\]
It is easy to check that $\Theta_3^\tau\subseteq \Theta_3^*0^\infty$. For $I\in\Theta_3^\tau$, we define $\tau^*(I)$ by
\[
\tau^*(I)=\sum_{k=1}^\infty A^{k-1}\tau(I|_k)\in \mathbb{Z}^2.
\]

Now we can present our first result (Theorem \ref{max}), which  establishes a characterization of the maximal
orthogonal set of the measure $\mu_{A,\mathcal{D}}$ via  maximal tree mapping.

\begin{theorem}\label{thmma}
Let $\Lambda$ be a countable subset of $\mathbb{R}^2$ containing $\textbf{0}$. Then $\Lambda$ is a maximal orthogonal set of $\mu_{A,\mathcal{D}}$ if and only if there exists a maximal tree mapping $\tau$ such that $\Lambda=\tau^*(\Theta_3^\tau)$.
\end{theorem}
To prove it, we need the following simple lemma. We can write $\mu_{R,B}$  as the following infinite convolution of discrete measures
\begin{equation}\label{jj}
\begin{split}
\mu_{A,\mathcal{D}}&=\delta_{A^{-1}\mathcal{D}}\ast\delta_{A^{-2}\mathcal{D}}\ast\delta_{A^{-3}\mathcal{D}}\ast\cdots\\
&=\mu_n\ast\nu_{n},
\end{split}
\end{equation}
where
\[\mu_n=\delta_{A^{-1}\mathcal{D}}\ast\delta_{A^{-2}\mathcal{D}}\ast\cdots\ast\delta_{A^{-n}\mathcal{D}}\]
and $\nu_{n}=\delta_{A^{-(n+1)}\mathcal{D}}\ast\delta_{A^{-(n+2)}\mathcal{D}}\ast\cdots$.

\begin{lemma}\label{noss}
Suppose that $\Lambda$ is an orthogonal set of $\nu_n$. Then $\Lambda$ is an orthogonal set of $\nu_{n-1}$ but cannot be a maximal orthogonal set of $\nu_{n-1}$.
\end{lemma}
\begin{proof}
Since $\hat{\nu}_{n-1}(\xi)=\hat{\delta}_{A^{-n}\mathcal{D}}(\xi)\cdot\hat{\nu}_n(\xi)$ for all $\xi\in\mathbb{R}^2$, it follows that $\mathcal{Z}(\hat{\nu}_n)\subseteq\mathcal{Z}(\hat{\nu}_{n-1})$. Then $(\Lambda-\Lambda)\setminus\{\textbf{0}\}\subseteq\mathcal{Z}(\hat{\nu}_n)\subseteq\mathcal{Z}(\hat{\nu}_{n-1})$, which implies that $\Lambda$ is an orthogonal set of $\nu_{n-1}$.

Without loss of generality, we assume that $\textbf{0}\in\Lambda$. We next show that $\Lambda\cup\{w\}$ is an orthogonal set of $\nu_{n-1}$ for any $w\in\mathcal{Z}(\hat{\delta}_{A^{-n}\mathcal{D}})$. Therefore, $\Lambda$ cannot be a maximal orthogonal set of $\nu_{n}$. To this end, let $w\in\mathcal{Z}(\hat{\delta}_{A^{-n}\mathcal{D}})$. Then we can write $w=A^n\left(\frac{1}{3}\begin{pmatrix}
		 1 \\
		2\\
	\end{pmatrix}+a\right)$ or $A^n\left(\frac{2}{3}\begin{pmatrix}
		 1 \\
		2\\
	\end{pmatrix}+b\right)$ for some $a,b\in\mathbb{Z}^2$. For any $\lambda\in \Lambda \setminus \{\textbf{0}\}$, write $\lambda=A^{n+i}\left(\frac{1}{3}\begin{pmatrix}
		 1 \\
		2\\
	\end{pmatrix}+a_1\right)$ or $A^{n+i}\left(\frac{2}{3}\begin{pmatrix}
		 1 \\
		2\\
	\end{pmatrix}+b_1\right)$ for some $i\geq1$. When $w=A^n\left(\frac{1}{3}\begin{pmatrix}
		 1 \\
		2\\
	\end{pmatrix}+a\right)$ and $\lambda=A^{n+i}\left(\frac{2}{3}\begin{pmatrix}
		 1 \\
		2\\
	\end{pmatrix}+b_1\right)$. Then
\begin{equation*}
\begin{split}
\lambda-w&=A^n\left(\frac{2}{3}A^i\begin{pmatrix}
		 1 \\
		2\\
	\end{pmatrix}+A^ib_1-\frac{1}{3}\begin{pmatrix}
		 1 \\
		2\\
	\end{pmatrix}-a\right)\\
&=A^n\left(\frac{2}{3}A^i\begin{pmatrix}
		 1 \\
		2\\
	\end{pmatrix}-\frac{1}{3}\begin{pmatrix}
		 1 \\
		2\\
	\end{pmatrix}+A^ib_1-a\right).
\end{split}
\end{equation*}
Note that $\frac{2}{3}A^i\begin{pmatrix}
		 1 \\
		2\\
	\end{pmatrix}-\frac{1}{3}\begin{pmatrix}
		 1 \\
		2\\
	\end{pmatrix}\equiv\frac{2}{3}\begin{pmatrix}
		 1 \\
		2\\
	\end{pmatrix}\pmod{\mathbb{Z}^2}$ and $A^ib_1-a\in\mathbb{Z}^2$. Therefore, $\lambda-w\in A^n\left(\frac{2}{3}\begin{pmatrix}
		 1 \\
		2\\
	\end{pmatrix}+\mathbb{Z}^2\right)\subseteq \mathcal{Z}(\hat{\nu}_{n-1})$. The proofs of the other cases are similar and hence
\[
\left((\Lambda\cup\{w\})-(\Lambda\cup\{w\})\right)\setminus\{\textbf{0}\}\subseteq \mathcal{Z}(\hat{\nu}_{n-1}),
 \]
which implies that $\Lambda\cup\{w\}$ is an orthogonal set of $\nu_{n-1}$.
\end{proof}

\proof[Proof of Theorem  \ref{thmma}]
If there exists a maximal tree mapping $\tau$ such that $\Lambda=\tau^*(\Theta_3^\tau)$, for any two distinct words $I=i_1i_2\cdots,J=j_1j_2\cdots\in\Theta_3^\tau$, let $k$ be the smallest integer such that
$I|_k\neq J|_k$. Then by (ii) of Definition  \ref{deftree}, there exists $e_{I|_{k-1}}\in\varepsilon_{q_1}$ such that
\[\tau(I|_k)=e_{I|_{k-1}}+i_k\begin{pmatrix}
		q_1 \\
		-q_2 \\
	\end{pmatrix}\pmod{A}\]
and
\[\tau(J|_k)=e_{I|_{k-1}}+j_k\begin{pmatrix}
		q_1 \\
		-q_2 \\
	\end{pmatrix}\pmod{A}.\]
Here, if $k=1$, we define $e_\emptyset=\textbf{0}$ and the symbol $\pmod{A}$ in the above equations is not needed by $(i)$ in Definition  \ref{deftree}. It follows that there exists some $w\in\mathbb{Z}^2$ such that
\[\begin{split}
\tau^\ast(I)-\tau^\ast(J)&=A^{k-1}(\tau(I|_k)-\tau(J|_k)+Aw)\\
&=A^{k-1}\left((i_k-j_k)\begin{pmatrix}
		q_1 \\
		-q_2 \\
	\end{pmatrix}\pmod{A}\right).
\end{split}\]
Since $i_k-j_k\in\{\pm1,\pm2\}$, it follows that
\[(i_k-j_k)\begin{pmatrix}
		q_1 \\
		-q_2 \\
	\end{pmatrix}=\begin{pmatrix}
		q_1 \\
		2q_2 \\
	\end{pmatrix}\pmod{A}~ \text{or} ~\begin{pmatrix}
		2q_1 \\
		4q_2 \\
	\end{pmatrix}\pmod{A}.\]
Therefore, $\tau^\ast(I)-\tau^\ast(J)\in\mathcal{Z}(\hat{\mu}_{A,\mathcal{D}})$, which implies that $\Lambda=\tau^*(\Theta_3^\tau)$ is an orthogonal set of $\mu_{A,\mathcal{D}}$.

Now, we prove the maximality by contradiction. If there exists $\gamma\notin\tau^\ast(\Theta_3^\tau)$ such that $\tau^*(\Theta_3^\tau)\cup\{\gamma\}$ is also an orthogonal set of $\mu_{A,\mathcal{D}}$, then $\gamma=\gamma-\textbf{0}\in\mathcal{Z}(\hat{\mu}_{A,\mathcal{D}})\subseteq\mathbb{Z}^2$ since $\textbf{0}=\tau^\ast(0^\infty)\in \Lambda$. By Proposition \ref{exp} we can write
\[\gamma=\sum_{k=1}^\infty A^{k-1}c_k\]
with  all $c_k\in\Gamma_{q_1,q_2}$ and $c_k=\textbf{0}$ for $k$ large enough. Then, by assumption we have
\[
c_1c_2\cdots\in\Gamma_{q_1,q_2}^\infty\setminus\{\tau(I|_1)\tau(I|_2)\cdots:I\in\Theta_3^\tau\},
\]
where $\Gamma_{q_1,q_2}^\infty:=\Gamma_{q_1,q_2}\times \Gamma_{q_1,q_2}\cdots.$
Let $s$ be the first index such that the sequence $c_1c_2\cdots c_s\notin\{\tau(I|_1)\tau(I|_2)\cdots \tau(I|_s):I\in\Theta_3^\tau\}$. Then there exists $I_0=i_1i_2\cdots\in\Theta_3^\tau$ such that
 $\tau(I_0|_i)=c_i$ for $1\leq i\leq s-1$ and
 \[\tau^\ast(I_0)-\gamma=A^{s-1}(\tau(I_0|_s)-c_s+Aw)\in\mathcal{Z}(\hat{\mu}_{A,\mathcal{D}})\]
 for some $w\in\mathbb{Z}^2$. According to the value range of $\tau(I_0|_s)$ and $c_s$, we have $A\nmid (\tau(I_0|_s)-c_s)$, i.e., there exists no integer vector $a$ such that $\tau(I_0|_s)-c_s=Aa$.
 Then $\tau(I_0|_s)-c_s+Aw\in\pm\begin{pmatrix}
		q_1 \\
		-q_2 \\
	\end{pmatrix}+A\mathbb{Z}^2$ and thus $\tau(I_0|_s)-c_s\in\pm\begin{pmatrix}
		q_1 \\
		-q_2 \\
	\end{pmatrix}+A\mathbb{Z}^2$. Moreover, for any $i\neq i_s$, by $(iii)$ in Definition \ref{deftree}, there exists $J_i\in\Theta_3^\ast$ and
 \[
 \tau^\ast(i_1i_2\cdots i_{s-1}iJ_i0^\infty)-\gamma=A^{s-1}(\tau(i_1i_2\cdots i_{s-1}i)-c_s+Aw')\in\mathcal{Z}(\hat{\mu}_{A,\mathcal{D}})
 \]
 for some $w'\in\mathbb{Z}^2$. By the similar analysis as above, we have $\tau(i_1i_2\cdots i_{s-1}i)-c_s\in\pm\begin{pmatrix}
		q_1 \\
		-q_2 \\
	\end{pmatrix}+A\mathbb{Z}^2$. On the other hand, by  (ii) in Definition \ref{deftree}, $\tau(i_1\cdots i_{s-1}i)-\tau(i_1\cdots i_{s})\in\pm\begin{pmatrix}
		q_1 \\
		-q_2 \\
	\end{pmatrix}+A\mathbb{Z}^2$ for any $i\neq i_s$. This, combined with the fact that the dimension of $L^2(\delta_{A^{-1}\mathcal{D}})$ is $3$, implies $\{\tau(i_1\cdots i_{s-1}i):i\in\Theta_3\}$ is a spectrum of $\delta_{A^{-1}\mathcal{D}}$.
Moreover, the above analysis implies that $\{\tau(i_1\cdots i_{s-1}i):i\in\Theta_3\}\cup \{c_s\}$ is an orthogonal set of $\delta_{A^{-1}\mathcal{D}}$, which contradicts the fact that $\{\tau(i_1\cdots i_{s-1}i):i\in\Theta_3\}$ is a spectrum of $\delta_{A^{-1}\mathcal{D}}$.

Conversely, suppose that $\Lambda$ is maximal orthogonal set of $\mu_{A,\mathcal{D}}$ containing zero. Then $\Lambda\setminus\{0\}\subseteq\bigcup_{j=1}^{\infty}A^j\left(\Bigg(\frac{1}{3}\begin{pmatrix}
		1 \\
		2 \\
	\end{pmatrix}+\mathbb{Z}^2\bigg)\cup\bigg(\frac{2}{3}\begin{pmatrix}
		1 \\
		2 \\
	\end{pmatrix}+\mathbb{Z}^2\Bigg)\right)\subseteq\mathbb{Z}^2$. Write $\Lambda=\{\lambda_n\}_{n=0}^{\infty}$ with $\lambda_0=\textbf{0}$. By Proposition \ref{exp} again, we can expand $\lambda_n$ in $A$-adic expansion with digits chosen from $\Gamma_{q_1,q_2}$ as follows
\[\lambda_n=\sum_{k=1}^{\infty}{A^{k-1}c_{n,k}},\]
where $c_{n,k}=\textbf{0}$ for $k$ sufficiently large.

Let $\mathcal{A}(\emptyset)=\{c_{n,1}:n\geq0\}$. We claim that $\mathcal{A}(\emptyset)=\mathcal{C}_{q_1,q_2}$. In fact, first, it is easy to see that $0\in\mathcal{A}(\emptyset)$. Then we claim that $\#\mathcal{A}(\emptyset)\geq2$. If it is not true, then $\mathcal{A}(\emptyset)=\{\textbf{0}\}$. Note that $\mu_{A,\mathbb{D}}=\delta_{A^{-1}\mathbb{D}}\ast\delta_{A^{-2}\mathbb{D}}\ast\cdots=:\delta_{A^{-1}\mathcal{D}}\ast \nu$, where $\nu=\delta_{A^{-2}\mathcal{D}}\ast\delta_{A^{-3}\mathcal{D}}\ast\cdots$. Then $\Lambda\setminus\{\textbf{0}\}\subseteq\mathcal{Z}(\hat{\nu})$, which implies that $\Lambda$ is an orthogonal set of $\nu$. By Lemma \ref{noss}, we have that $\Lambda$ cannot be a maximal orthgonal set of $\mu_{A,\mathbb{D}}$, which contradicts the assumption that $\Lambda$ is a maximal orthogonal set of $\mu_{A,\mathcal{D}}$. Then $\#\mathcal{A}(\emptyset)\geq2$.

 Next, we prove that $\#\mathcal{A}(\emptyset)\leq3$. For any $\omega_1,\omega_2\in\mathcal{A}(\emptyset)$ with $\omega_1\neq\omega_2$, there exist $m,n$ such that $w_1=c_{m,1}\neq c_{n,1}=w_2$. Since $\Lambda$ is an orthogonal set of $\mu_{A,\mathcal{D}}$, it follows that there exists $a\in\mathbb{Z}^2$ such that
\[\lambda_m-\lambda_n=c_{m,1}-c_{n,1}+Aa\in\mathcal{Z}(\hat{\mu}_{A,\mathcal{D}}).\]
By $c_{m,1},c_{n,1}\in\Gamma_{q_1,q_2}$, we have $\lambda_m-\lambda_n\in\mathcal{Z}(\hat{\delta}_{A^{-1} D})$ and thus $w_1-w_2\in\mathcal{Z}(\hat{\delta}_{A^{-1} D})$. This implies that $\mathcal{A}(\emptyset)$ is an orthogonal set of $\delta_{A^{-1}\mathcal{D}}$ and thus $\#\mathcal{A}(\emptyset)\leq3$. If $2\leq\mathcal{A}(\emptyset)<3$, write $\mathcal{A}(\emptyset)=\{0,w\}$. Then $w=w-0\in\mathcal{Z}(\hat{\delta}_{A^{-1} \mathcal{D}})$. Combined with the fact that $w\in\Gamma_{q_1,q_2}$, we have $w=\begin{pmatrix}
		q_1 \\
		-q_2 \\
	\end{pmatrix}$ or $-\begin{pmatrix}
		q_1 \\
		-q_2 \\
	\end{pmatrix}$. Since $\begin{pmatrix}
		q_1 \\
		-q_2 \\
	\end{pmatrix}-\left(-\begin{pmatrix}
		q_1 \\
		-q_2 \\
	\end{pmatrix}\right)=\begin{pmatrix}
		2q_1 \\
		-2q_2 \\
	\end{pmatrix}\equiv\begin{pmatrix}
		-q_1 \\
		q_2 \\
	\end{pmatrix}\pmod{A}$, for any $\lambda\in\Lambda$, it follows that $\lambda-(-w)=w~\text{or}~2w\pmod{A}\in\mathcal{Z}(\hat{\delta}_{A^{-1} \mathcal{D}})\subseteq\mathcal{Z}(\hat{\mu}_{A,\mathcal{D}})$   and thus $\Lambda\cup\{-w\}$ is an orthogonal set of $\mu_{A,\mathcal{D}}$. This contradicts the fact that $\Lambda$ is a maximal orthogonal set of $\mu_{A,\mathcal{D}}$. Hence, we complete the proof of the above claim.

Now, we can define the first level of maximal tree mapping. Define $\tau(i)=i\begin{pmatrix}
		q_1 \\
		-q_2 \\
	\end{pmatrix}$ for $i\in\Theta_3$. Then we can define
\[\mathcal{A}(i)=\{c_{n,2}:c_{n,1}=\tau(i),n\geq0\}.\]
Then we can write
\[\Lambda=\bigcup_{i=-1}^{1}(\tau(i)+A\Lambda_i),\]
where $\Lambda_i=A^{-1}(\Lambda-\tau(i))$. Note that $\Lambda_i$ is  a maximal orthogonal set of $\mu_{A,\mathcal{D}}$. In fact,
\begin{equation*}
\begin{split}
	 (\Lambda_{i}-\Lambda_{i})\setminus\{\mathbf{0}\}  &\subseteq \left(A^{-1}(\Lambda-\Lambda)\setminus\{\mathbf{0}\}\right) \cap \mathbb{Z}^2\\
&\subseteq \left(A^{-1}\bigcup_{k=1}^{\infty}A^k\mathcal{Z}(m_\mathcal{D}) \right)\cap \mathbb{Z}^2\\
&=\bigcup_{k=1}^{\infty}A^k\mathcal{Z}(m_\mathcal{D}).
\end{split}
	 \end{equation*}
This implies that $\Lambda_i$ is an orthogonal set of $\mu_{A,\mathcal{D}}$. Furthermore, assume that there exists $\beta\notin\Lambda_i$ such that $\Lambda_i\cup\{\beta\}$ is also an orthogonal set of $\mu_{A,\mathcal{D}}$. Let $\alpha=\tau(i)+A\beta$.
 It follows from the definition of $\Lambda_i$ that $\alpha\notin\Lambda$. For any $\lambda\in\Lambda$, write $\lambda=\tau(j)+A\beta_j$ for some $j\in\Theta$ and $\beta_j\in\Lambda_j$. If $i=j$, then $\alpha-\lambda=A(\beta-\beta_i)\in\mathcal{Z}(\hat{\mu}_{A,\mathcal{D}})$. If $i\neq j$, then $\alpha-\lambda=(\tau(i)-\tau(j))+A(\beta-\beta_j)\in\mathcal{Z}(\hat{\mu}_{A,\mathcal{D}})$. Therefore, $\hat{\mu}_{A,\mathcal{D}}(\alpha-\lambda)=0$ for any $\lambda\in\Lambda$, which contradicts the maximality of $\Lambda$.

Then we can do similar analysis to $\mathcal{A}(i)$ as $\mathcal{A}(\emptyset)$ by regarding $\Lambda_i$ as $\Lambda$. By induction, for $I=i_1i_2\cdots i_k$, we define
\[\mathcal{A}(I)=\{c_{n,k+1}:c_{n,1}=\tau(i_1),c_{n,2}=\tau(i_1i_2),\ldots,c_{n,k}=\tau(I),n\geq0\}.\]
Note that $0$ may be not in $\mathcal{A}(I)$ and similarly, we have $\mathcal{A}(I)=\mathcal{C}_{q_1,q_2}+a\pmod{A}$ for some $a\in\varepsilon_{q_1}$.  Define $e_I=a$ and $\tau(Ii)=a+i\begin{pmatrix}
		q_1 \\
		-q_2 \\
	\end{pmatrix}\pmod{A}$ for $i\in\Theta_3$.

Next, we show that $\tau$ is a maximal tree mapping from $\Theta_3^*$ to $\Gamma_{q_1,q_2}$. By the above construction, conditions $(i)$ and $(ii)$ in Definition \ref{deftree} are satisfied trivially. Finally, we check  the condition $(iii)$. For any $I=i_1i_2\cdots i_k\in\Theta_3^k$, choose $\lambda_n$ with
\[\lambda_n=\sum_{j=1}^{\infty}A^{j-1}c_{n,j}\]
such that
\[c_{n,1}=\tau(i_1),c_{n,2}=\tau(i_1i_2),\ldots,c_{n,k}=\tau(I).\]
Since $c_{n,j}$ equals $\textbf{0}$ for all sufficiently large $j$, we continue the construction by following the digit expansion of $\lambda_n$. This implies that $\tau$ satisfies $(iii)$ in the definition of the maximal tree mapping.
\qed

\section{Regular spectra of $\mu_{A,\mathcal{D}}$}\label{rege}
In this section, we restrict our attention to a class of particular maximal tree mappings, which can produce a class of spectra for $\mu_{A,\mathcal{D}}$ under some conditions.
\begin{defn}\label{remal}
 We say that a mapping $\tau$ from $\Theta_3^*$ to $\Gamma_{q_1,q_2}$ is a {\it regular mapping} if
\begin{enumerate}[(i)]
\item  $\tau(0^ki_{k+1})=i_{k+1}\begin{pmatrix}
		q_1 \\
		-q_2 \\
	\end{pmatrix}$ for  $k\ge 0$ and $i_{k+1}\in\Theta_3$;
  \item  For any $k\ge 1$ and  $Ij\in\Theta_3^{k+1}$ with  $I\ne 0^k, j\in\Theta_3$, $\tau(Ij)\equiv e_I+j\begin{pmatrix}
		q_1 \\
		-q_2 \\
	\end{pmatrix}\pmod{A}$ for some $e_I\in\mathcal{E}_{q_1}$;
\item  For any $I\in\Theta_3^*$ and sufficiently large $n$,  $\tau(I0^n)=\textbf{0}$.
\end{enumerate}
\end{defn}

\begin{remark}
Regular mapping can give a natural ordering of the maximal orthogonal set. In regular case,
\[\Theta_3^\tau=\{I0^\infty:I\in \Theta_3^\ast\}\]
and the set $\tau^*(\Theta_3^\tau)$ can be enumerated as follows. For $I=i_1i_2\cdots i_n\in\Theta_3^n\setminus\Theta_3^{n-1}0$, there exists a unique integer $k_I$ such that
\[
k_I=i_1+i_2 3+\cdots+i_n 3^{n-1}.
\]
On the other hand, it is easy to check that
\[
\mathbb{Z}\setminus\{0\}=\bigcup_{n=1}^\infty\bigcup_{I\in\Theta_3^n\setminus\Theta_3^{n-1}0} k_I.
\]
Let $\lambda_0=\textbf{0}$ and for $k\in \mathbb{Z}\setminus\{0\}$, define $\lambda_k=\tau^*(i_1i_2\cdots i_n0^\infty)$ if
\begin{equation}\label{eqwwd}
k=i_1+i_2 3+\cdots+i_n 3^{n-1}\quad \text{with $i_j\in\Theta_3$ and $i_n\neq0$}.
\end{equation}
Then $\tau^*(\Theta_3^\tau)=\{\lambda_k\}_{k\in\mathbb{Z}}$.
\end{remark}

For  $k\ge 1$, define \[\ell_k=\#\{i:\tau(I0^i)\neq \textbf{0},i\geq1\},\]
where $I=i_1i_2\cdots i_n$ is defined by \eqref{eqwwd}. It follows from  Theorem \ref{max} that $\tau^*(\Theta_3^\tau)$ is a maximal orthogonal set of $\mu_{A,\mathcal{D}}$. Under the uniformly bounded condition of $\ell_k$, we have the following (Theorem \ref{ppd})

\begin{theorem}\label{ppd112}
Let $\tau$ be a regular mapping and $\Lambda(\tau)=\tau^*(\Theta_3^\tau)$. If
\begin{equation}\label{condwe}
\ell_{\max}:=\max_{k\geq 1}\{\ell_k\}<\infty,
\end{equation}
then $\Lambda(\tau)$ is a spectrum of $\mu_{A,\mathcal{D}}$.
\end{theorem}
\noindent For simplicity, if $\ell_{\max}<\infty$, we call $\Lambda(\tau)$ a {\it regular spectrum} of $\mu_{A,\mathcal{D}}$.

\begin{example}
For any $I=i_1i_2\cdots i_k\in\Theta_3^k,k\geq 1$, let $\tau(I)=i_k\begin{pmatrix}
		q_1 \\
		-q_2 \\
	\end{pmatrix}$. Then $\tau$ is a regular mapping from $\Theta_3^\ast$ to $\Gamma_{q_1,q_2}$ and
\begin{equation*}
\text{$\Lambda(\tau)=\bigcup_{k=1}^{\infty}\Lambda_k$,~~ where $\Lambda_k=\mathcal{L}+A \mathcal{L}+\cdots+A^{k-1}\mathcal{L}$}
\end{equation*}
and
\begin{equation}\label{do}
\mathcal{L}=\left\{\begin{pmatrix}
		 0 \\
		0  \\
	\end{pmatrix},\begin{pmatrix}
		 q_1 \\
		-q_2  \\
	\end{pmatrix},\begin{pmatrix}
		 -q_1 \\
		q_2  \\
	\end{pmatrix}\right\}.
\end{equation}
By Theorem \ref{ppd112}, we have $\Lambda(\tau)$ is a spectrum of $\mu_{A,\mathcal{D}}$. This is also proved by Li \cite{JLL10}.
\end{example}

\proof[Proof of Theorem \ref{ppd112}] We will prove the theorem by trying to check the following Jorgensen and Pederson's criteria:

\begin{theorem}[\cite{JP98}]\label{thjp}
 Let $\mu$ be a Borel probability measure with compact support in $\mathbb{R}^2$, and let $\Lambda\subseteq \mathbb{R}^2$ be a countable subset. Then
 \begin{enumerate}
 \item [\textup{(i)}]  $\Lambda$
    is an orthonormal set of $\mu$ if and only if $Q_\Lambda(\xi):=\sum_{\lambda\in\Lambda}|\hat{\mu}(\xi+\lambda)|^2\le 1$
    for $\xi\in\mathbb{R}^2$. In this case $Q_\Lambda(z)$ is an entire function in $\mathbb{C}^2$;
 \item [\textup{(ii)}] $\Lambda$ is a spectrum of $\mu$ if and only if $Q_\Lambda(\xi)\equiv
    1$ for $\xi\in\mathbb{R}^2$.
 \end{enumerate}
 \end{theorem}

    For simplicity, we write $\mu:=\mu_{A,\mathcal{D}}$. For $n\ge 1$, let $\alpha_n=\frac{3^n-1}{2}$. Then each $\alpha_n$ is an integer.  Define
\[
Q_n(\xi)=\sum_{\lambda\in \Lambda_{\alpha_n}}\left|\hat{\mu}(\xi+\lambda)\right|^2, ~~Q_{\Lambda(\tau)}(\xi)=\sum_{\lambda\in \Lambda(\tau)}\left|\hat{\mu}(\xi+\lambda)\right|^2,
\]
where $\Lambda_{\alpha_n}=\{\lambda_k\}_{k=-\alpha_n}^{\alpha_n}$ and $\Lambda(\tau)=\tau^\ast(\Theta_3^\tau)=\{\lambda_k\}_{k\in \mathbb{Z}}$. Then, for any $n,s\ge 1$, by \eqref{jj} we have the following indentity:
\begin{equation}\label{fdgj}
\begin{split}
Q_{n+s}(\xi)&=Q_n(\xi)+\sum_{\lambda\in \Lambda_{\alpha_{n+s}}\setminus \Lambda_{\alpha_n}}\left|\hat{\mu}(\xi+\lambda)\right|^2\\
            &=Q_n(\xi)+\sum_{\lambda\in \Lambda_{\alpha_{n+s}}\setminus \Lambda_{\alpha_n}}|\hat{\mu}_{n+s}(\xi+\lambda)|^2|\hat{\nu}_{n+s}(\xi+\lambda)|^2.
\end{split}
\end{equation}

Since $Q_\Lambda$ is an entire function by Theorem \ref{thmma} and $(ii)$ in Theorem \ref{thjp}, we just need to consider some $\xi$ around $\textbf{0}$.

\begin{lemma}\label{lemss}
For $n\ge 1$, the measure $\mu_n$ is a spectral measure with a spectrum $\Lambda_{\alpha_n}=\{\lambda_k\}_{k=-\alpha_n}^{\alpha_n}$.
\end{lemma}
\begin{proof}
For any two distinct integers $k,k'$ with $-\alpha_n\leq k,k'\leq \alpha_n$, there exist two distinct words $I,I'\in\Theta_3^n$ such that
\[\lambda_k=\tau^\ast(I0^\infty),~\lambda_{k'}=\tau^\ast(I'0^\infty).\]
Let $s~(\leq n)$ be the first index such that $I|_s\neq I'|_s$. Then there exists integer $M$ such that
\[
\lambda_k-\lambda_{k'}=A^{s-1}(\tau(I|_s)-\tau(I'|_s)+AM).
\]
By the definition of the regular mapping $\tau$, we have that $\lambda_k-\lambda_{k'}\in\mathcal{Z}(\hat{\mu}_s)\subseteq \mathcal{Z}(\hat{\mu}_n)$. This implies that $\Lambda_{\alpha_n}$ is an orthogonal set for $\mu_n$. On the other hand, the dimension of $L^2(\mu_n)$ is $3^n=\#\Lambda_{\alpha_n}$. So, we complete the proof.
\end{proof}

\begin{lemma}\label{lemssp}
    Suppose $\ell_{\max}<\infty$. Let  $\xi=\begin{pmatrix}
		\xi^{(1)} \\
		\xi^{(2)} \\
	\end{pmatrix}\in \mathbb{R}^2$ with $|\xi^{(1)}|\leq \frac{3q_1}{2(3q_1-1)}$ and $|\xi^{(2)}|\leq \frac{3q_2}{2(3q_2-1)}$. Then there exists $c>0$ such that
    \begin{equation*}
    \inf_{\lambda\in\Lambda_{\alpha_{n+s}}\setminus \Lambda_{\alpha_n}}|\hat{\nu}_{n+s}(\xi+\lambda)|^2\geq c>0.
    \end{equation*}
\end{lemma}
    \begin{proof}
    Note that $\Lambda_{\alpha_{n+s}}\setminus \Lambda_{\alpha_n}=\{\lambda_k\}_{\alpha_n<|k|\leq \alpha_{n+s}}$. For any integer $k$ with $\alpha_n<|k|\leq \alpha_{n+s}$, there exists an integer  $N$ with $n<N\leq n+s$ such that
    \[\alpha_{N-1}<|k|\leq \alpha_{N}.\]
    Also, there exists a word $I=i_1i_2\cdots i_N\in \Theta_3^N$ such that
    \[k=i_1+i_2 3+\cdots+i_N 3^{N-1}\]
    with $i_N\neq0$.

    Let  $\xi=\begin{pmatrix}
		\xi^{(1)} \\
		\xi^{(2)} \\
	\end{pmatrix}\in \mathbb{R}^2$ with $|\xi^{(1)}|\leq \frac{3q_1}{2(3q_1-1)}$ and $|\xi^{(2)}|\leq \frac{3q_2}{2(3q_2-1)}$. Then
\[
\begin{split}
\xi+\lambda_k&=\xi+\tau(I|_1)+A\tau(I|_2)+\cdots+A^{N-1}\tau(I|_N)+A^N\tau(I0)+\\
&\quad\quad\quad\cdots+A^{n+s-1}\tau(I0^{n+s-N})+\sum_{i=1}^{\infty}A^{n+s-1+i}\tau(I0^{n+s-N+i})
\end{split}
\]
and thus
\[
\begin{split}
A^{-(n+s)}(\xi&+\lambda_k)=A^{-(n+s)}\big(\xi+\tau(I|_1)+A\tau(I|_2)+\cdots+A^{N-1}\tau(I|_N)\\
& +A^N\tau(I0)+\cdots+A^{n+s-1}\tau(I0^{n+s-N})\big)+\sum_{i=1}^{\infty}A^{i-1}\tau(I0^{n+s-N+i})\\
&\quad\quad\quad=:\xi_0+\sum_{i=1}^{\infty}A^{i-1}\tau(I0^{n+s-N+i}).
\end{split}
\]
Write $\xi_0=\begin{pmatrix}
		\xi_0^{(1)} \\
		\xi_0^{(2)} \\
	\end{pmatrix}.$ Then
\[
\begin{split}
|\xi_0^{(1)}|&\leq \frac{1}{(3q_1)^{n+s}}\left(|\xi^{(1)}|+\frac{3q_1}{2}+3q_1\cdot\frac{3q_1}{2}+\cdots+(3q_1)^{n+s-1}\cdot\frac{3q_1}{2}\right)\\
&\leq \frac{1}{(3q_1)^{n+s}}\left(\frac{3q_1}{2(3q_1-1)}+\frac{3q_1}{2}+3q_1\cdot\frac{3q_1}{2}+\cdots+(3q_1)^{n+s-1}\cdot\frac{3q_1}{2}\right)\\
&=\frac{3q_1}{2(3q_1-1)}.
\end{split}
\]
Similarly, we have $|\xi_0^{(2)}|\leq \frac{3q_2}{2(3q_2-1)}$.

 Write
\[
\{i:\tau(I0^{n+s-N+i})\neq \textbf{0}\}=\{i_1,i_2,\ldots,i_{L_0}\}
\]
with $i_1<i_2<\cdots<i_{L_0}$
and define $n_{i_j}:=n+s-N+i_j$ for $1\leq j\leq L_0$. By the assumption \eqref{condwe}, we have that $L_0\leq \ell_{\max}$.
Note that $\hat{\nu}_m(x)=\hat{\mu}(A^{-m}x)$ and $\hat{\mu}(x)=\hat{\mu}_m(x)\hat{\mu}(A^{-m}x)$ for all integer $m\geq1$ and $x\in\mathbb{R}^2$.
Then, by the periodicity of the functions $\hat{\delta}_{A^{-m}\mathcal{D}}$ we have
\begin{eqnarray}\label{eqduitui}
\begin{split}
|\hat{\nu}_{n+s}(\xi&+\lambda_k)|=|\hat{\mu}(A^{-(n+s)}(\xi+\lambda_k))|\\
&=\left|\hat{\mu}\bigg(\xi_0+\sum_{j=1}^{L_0}A^{i_j-1}\tau(I0^{n_{i_j}})\bigg)\right|\\
&=\left|\hat{\mu}_{i_1-1}\bigg(\xi_0+\sum_{j=1}^{L_0}A^{i_j-1}\tau(I0^{n_{i_j}})\bigg)\right|\cdot\\
&\left|\hat{\mu}\bigg(A^{-(i_1-1)}\xi_0+A^{-(i_1-1)}\sum_{j=1}^{L_0}A^{i_j-1}\tau(I0^{n_{i_j}})\bigg)\right|\\
&=|\hat{\mu}_{i_1-1}(\xi_0)|\cdot\left|\hat{\mu}\bigg(A^{-(i_1-1)}\xi_0+A^{-(i_1-1)}\sum_{j=1}^{L_0}A^{i_j-1}\tau(I0^{n_{i_j}})\bigg)\right|\\
&= |\hat{\mu}_{i_1-1}(\xi_0)|\cdot|\hat{\mu}_{i_2-i_1}(A^{-(i_1-1)}\xi_0+\tau(I0^{n_{i_1}}))|\cdot\\
&\quad\quad\left|\hat{\mu}\left(A^{-(i_2-i_1)}\bigg(A^{-(i_1-1)}\xi_0+A^{-(i_1-1)}\sum_{j=1}^{L_0}A^{i_j-1}\tau(I0^{n_{i_j}})\bigg)\right)\right|\\
&\geq |\hat{\mu}(\xi_0)|\cdot|\hat{\mu}(A^{-(i_1-1)}\xi_0+\tau(I0^{n_{i_1}}))|\cdot\\
&\quad\quad \left|\hat{\mu}\left(A^{-(i_2-i_1)}\bigg(A^{-(i_1-1)}\xi_0+A^{-(i_1-1)}\sum_{j=1}^{L_0}A^{i_j-1}\tau(I0^{n_{i_j}})\bigg)\right)\right|.
\end{split}
\end{eqnarray}

 Write $A^{-(i_1-1)}\xi_0+\tau(I0^{n_{i_1}})=\begin{pmatrix}
		a^{(1)} \\
		a^{(2)} \\
	\end{pmatrix}$. We claim that $|a^{(1)}|<q_1 $ and $|a^{(2)}|<q_2 $. In fact, if $i_1=1$, when $q_1=1$, then $\tau(I0^{n_{i_1}})=\textbf{0}$ and thus $|a^{(1)}|\leq \frac{3q_1}{2(3q_1-1)}=\frac{3}{4}$. Similarly, $|a^{(1)}|\leq \frac{3q_2}{2(3q_2-1)}$. When $q_1\geq 2$, $|a^{(1)}|\leq \frac{3q_1}{2(3q_1-1)}+\frac{q_1}{2}<q_1 $. Similarly, $|a^{(2)}|<q_2$. If $i_1\geq 2$, then $|a^{(1)}|<q_1 $ and $|a^{(2)}|<q_2 $.

Next, we consider the third term $$\left|\hat{\mu}\left(A^{-(i_2-i_1)}\bigg(A^{-(i_1-1)}\xi_0+A^{-(i_1-1)}\sum_{j=1}^{L_0}A^{i_j-1}\tau(I0^{n_{i_j}})\bigg)\right)\right|$$ in $\eqref{eqduitui}$. Note that
\[
\begin{split}
\Bigg|\hat{\mu}&\bigg(A^{-(i_2-i_1)}\big(A^{-(i_1-1)}\xi_0+A^{-(i_1-1)}\sum_{j=1}^{L_0}A^{i_j-1}\tau(I0^{n_{i_j}})\big)\bigg)\Bigg|\\
&=\left|\hat{\mu}\left(A^{-(i_2-1)}\xi_0+A^{-(i_2-i_1)}\tau(I0^{n_{i_1}})+A^{-(i_2-1)}\sum_{j=2}^{L_0}A^{i_j-1}\tau(I0^{n_{i_j}})\right)\right|.
\end{split}
\]
Let $\xi_1=\begin{pmatrix}
		\xi_1^{(1)} \\
		\xi_1^{(2)} \\
	\end{pmatrix}=A^{-(i_2-1)}\xi_0+A^{-(i_2-i_1)}\tau(I0^{n_{i_1}})$. Then it is easy to see that $|\xi_1^{(1)}|\leq \frac{3q_1}{2(3q_1-1)}<q_1$ and $|\xi_1^{(2)}|\leq \frac{3q_2}{2(3q_2-1)}<q_2$.

Define
\[
C_\tau:=\min_{x\in S}|\hat{\mu}(x)|,
\]
where
\[
S=\left\{x=\begin{pmatrix}
		x^{(1)} \\
		x^{(2)} \\
	\end{pmatrix}: |x^{(1)}|\le q_1, |x^{(2)}|\le q_2\right\}.
\]
It follows from $\eqref{eqmle33}$ that the compact set $S$  does not intersect with the zero set of $\hat{\mu}(x)$. This implies that $0<c_\tau\le 1$. Hence, by  $\eqref{eqduitui}$ and the above argument we have
 \[
 |\hat{\nu}_{n+s}(\xi+\lambda_k)|\geq |\hat{\mu}(\xi_0)|\cdot C_\tau\cdot \left|\hat{\mu}\left(\xi_1+A^{-(i_2-1)}\sum_{j=2}^{L_0}A^{i_j-1}\tau(I0^{n_{i_j}})\right)\right|
 \]
 By induction, we have that
 \[
 |\hat{\nu}_{n+s}(\xi+\lambda_k)|\geq |\hat{\mu}(\xi_0)| \cdot c_\tau^{L_0+1}\geq |\hat{\mu}(\xi_0)|\cdot c_\tau^{\ell_{\max}+1}=:c.
 \]
 Hence, we complete the proof.
 \end{proof}

Now we return to the proof of Theorem \ref{ppd112} and the rest argument is standard.

Let $\xi=\begin{pmatrix}
		\xi^{(1)} \\
		\xi^{(2)} \\
	\end{pmatrix}\in \mathbb{R}^2$ with $|\xi^{(1)}|\leq \frac{3q_1}{2(3q_1-1)}$ and $|\xi^{(2)}|\leq \frac{3q_2}{2(3q_2-1)}$.
It follows from Theorem \ref{thjp} and Lemma \ref{lemss} that $\sum_{\lambda\in \Lambda_{\alpha_{n+s}}}\left|\hat{\mu}_{n+s}(\xi+\lambda)\right|^2=1$. Hence,
by \eqref{fdgj} and Lemma \ref{lemssp} we have
\[
\begin{split}
Q_{n+s}(\xi)&\ge Q_{n}(\xi)+c\sum_{\lambda\in \Lambda_{\alpha_{n+s}}\setminus \Lambda_{\alpha_n}}\left|\hat{\mu}_{n+s}(\xi+\lambda)\right|^2\\
                    &=Q_{n}(\xi)+c\left(1-\sum_{\lambda\in \Lambda_{\alpha_n}}\left|\hat{\mu}_{n+s}(\xi+\lambda)\right|^2\right).
\end{split}
\]

Letting $s\to \infty$, the above inequality becomes
\[
Q_{\Lambda}(\xi)\ge Q_{n}(\xi)+c(1-Q_{n}(\xi)).
\]
Then, letting $n\to \infty$ we have
\[
0\ge c(1-Q_{\Lambda}(\xi)),
\]
which implies that $Q_{\Lambda}(\xi)\ge 1$. On the other hand, it follows from the orthogonality of $\Lambda(\tau)$ and Theorem \ref{thjp} that $Q_{\Lambda}(\xi)\le 1$. Therefore, $Q_{\Lambda}(\xi)= 1$, and by Theorem \ref{thjp} again, $\Lambda(\tau)$ is a spectrum of $\mu.$
\qed

\section{Beurling dimension of spectra of $\mu_{A,\mathcal{D}}$}\label{rege1}

This section is devoted to the study of the Beurling dimension of the spectra of $\mu_{A,\mathcal{D}}$ and  the proof of Theorem \ref{thexu} will be presented.

\begin{prop}\label{mulp}
Suppose that $\Lambda$ is an orthogonal set of $\mu_{A,\mathcal{D}}$. Then $\Lambda$ has at most one point on every horizontal and vertical line.
\end{prop}
\begin{proof}
For any two distinct elements $\lambda_1,\lambda_2$ in $\Lambda$, by the orthogonality of $\Lambda$, we have
\[\lambda_1-\lambda_2\in\mathcal{Z}(\hat{\mu}_{A,\mathcal{D}})=\bigcup_{j=1}^{\infty}A^j\left(\Bigg(\frac{1}{3}\begin{pmatrix}
		1 \\
		2 \\
	\end{pmatrix}+\mathbb{Z}^2\bigg)\cup\bigg(\frac{2}{3}\begin{pmatrix}
		1 \\
		2 \\
	\end{pmatrix}+\mathbb{Z}^2\Bigg)\right).\]
Then there exists $k\geq1$ such that
\begin{equation}\label{eqhla}
\lambda_1-\lambda_2\in A^k\left(\Bigg(\frac{1}{3}\begin{pmatrix}
		1 \\
		2 \\
	\end{pmatrix}+\mathbb{Z}^2\bigg)\cup\bigg(\frac{2}{3}\begin{pmatrix}
		1 \\
		2 \\
	\end{pmatrix}+\mathbb{Z}^2\Bigg)\right).
\end{equation}
Write $\lambda_1=\begin{pmatrix}
		\lambda_1^{(1)} \\
		\lambda_1^{(2)} \\
	\end{pmatrix}$ and $\lambda_2=\begin{pmatrix}
		\lambda_2^{(1)} \\
		\lambda_2^{(2)} \\
	\end{pmatrix}$. By \eqref{eqhla}, we have $\lambda_1^{(1)}\neq\lambda_2^{(1)}$ and $\lambda_1^{(2)}\neq\lambda_2^{(2)}$. This implies that $\Lambda$ has at most one point on every horizontal and vertical line.
\end{proof}

Let $\pi_x$ and $\pi_y$ be the canonical projections of $\R^2$ onto the $x$ and $y$-axes, respectively.
The following proposition, which is proved in \cite{DFY21}, establishes a relationship of orthogoanl sets between $\mu_{A,\mathcal{D}}$ and the self-similar measures $\mu_{3q_i,\{0,1,2\}},i=1,2$. Here, for any integer $q\geq1$, $\mu_{3q,\{0,1,2\}}$ is the unique Borel probability measure $\mu:=\mu_{3q, \{0,1,2\}}$, which satisfies the invariance equation
    \begin{eqnarray*}
    \mu(E)=\frac 1{3}\sum_{i=0}^{2}\mu(3qE-i),\qquad \text{for any Borel set $E$}.
    \end{eqnarray*}
    \begin{prop}[\cite{DFY21}]\label{projj}
Let $\Lambda$ be an orthogonal set of $\mu_{A,\mathcal{D}}$. Then $\pi_x(\Lambda)$ is an orthogonal set of $\mu_{3q_1,\{0,1,2\}}$ and $\pi_y(\Lambda)$ is an orthogonal set of $\mu_{3q_2,\{0,1,2\}}$.
\end{prop}

Denote by
\[
\Lambda(A,\mathcal{L})=\bigcup_{k=1}^{\infty}\Lambda_k,
\]
where $\Lambda_k=\mathcal{L}+A \mathcal{L}+\cdots+A^{k-1}\mathcal{L}$ and $\mathcal{L}$ is given by \eqref{do}.
It is well-known that $\Lambda(A, \mathcal{L})$ is a spectrum of $\mu_{A,\mathcal{D}}$, see \cite{JLL10}. Moreover, by Theorem 1.4 in \cite{LWu322}, $\dim_{Be}(\Lambda(A,\mathcal{L}))=\frac{\log 3}{\log 3q_2}$.

The following result says that the value is in fact the optimal upper bound of Beurling dimensions of spectra for $\mu_{A,\mathcal{D}}$.

\begin{theorem}\label{thdddll}
For any  orthogonal set $\Lambda$ of $\mu_{A,\mathcal{D}}$, we have that
$\dim_{Be}(\Lambda)\leq\frac{\log3}{\log 3q_2}$. Moreover, suppose $\Lambda$ is a spectrum of  $\mu_{A,\mathcal{D}}$ and there exists an integer $p\geq1$ such that
\[
\sup_{\lambda\in\Lambda}\inf_{\gamma\in\Lambda}\|A^{-p}\lambda-\gamma\|<+\infty.
\]
Then $\dim_{Be}(\Lambda)=\frac{\log 3}{\log 3q_2}$.
\end{theorem}
\begin{proof}
Let $h>1$ and  $x=\begin{pmatrix}
		x^{(1)}\\
		x^{(2)}\\
	\end{pmatrix}\in\mathbb{R}^2$. By Proposition \ref{mulp} we have
\begin{equation}\label{eqpro}
\#(\Lambda\cap B(x,h))\le \#(\pi_y(\Lambda)\cap B(x^{(2)},h)).
\end{equation}
It follows from Proposition \ref{projj} that $\pi_y(\Lambda)$ is an orthogonal set of $\mu_{3q_2,\{0,1,2\}}$. So, by Theorem 3.5 in \cite{DHSW11}, we have that $\dim_{Be}(\pi_y(\Lambda))\leq \frac{\log 3}{\log (3q_2)}$. Fix $\varepsilon>0$. It follows from the  definition of Beurling dimension that
\[
\limsup\limits_{h\to \infty}\sup_{u\in \mathbb{R}}\frac{\#(\pi_y(\Lambda)\cap B(u,h))}{h^{\frac{\log 3}{\log (3q_2)}+\varepsilon}}=0.
\]
This, combined with \eqref{eqpro} and the definition of Beurling dimension, implies that $\dim_{Be}(\Lambda)\leq \frac{\log 3}{\log (3q_2)}$.

The second assertion follows immediately from Theorem 1.6 in \cite{TW21} and the first assertion.
\end{proof}

For any spectral measure, Shi \cite{Shi} proved that the Beurling dimension of its spectra is bounded by the upper entropy dimension of the spectral measure. For self-similar spectral measures or some Moran spectral measures, the upper bound can be attained, see \cite{DHSW11,HKTW18, LWu122}.  We next calculate the entropy dimension of $\mu_{A,\mathcal{D}}$ and our result (Proposition \ref{propeed}) says that the entropy dimension of $\mu_{A,\mathcal{D}}$ is strictly larger than $\frac{\log 3}{\log (3q_2)}$, which implies that the entropy dimension is not a good candidate to control the Beurling dimensions of spectra for $\mu_{A,\mathcal{D}}$.

Let's  recall the notion of entropy dimension for a probability measure $\mu$, which  is frequently used in fractal geometry and dynamical system. Let $\mathcal{P}_n$ be the $n$-th dyadic partition of $\mathbb{R}^d$, i.e.,
\[
\mathcal{P}_n=\{I_1\times I_2\times \cdots\times I_d:~I_j\in \mathcal{P}^{(1)}_n, 1\le j\le d\},
\]
where
\[
\mathcal{P}^{(1)}_n =\left\{\left[\frac{k}{2^n},\frac{k+1}{2^n}\right):k\in \mathbb{Z}\right\}.
\]
Write
\[
H_n(\mu)=-\sum_{Q\in \mathcal{P}_n}\mu(Q)\log \mu(Q).
\]

The \textit{upper entropy dimension} of $\mu$ is defined by
\[
\overline{\dim}_e\mu=\limsup\limits_{n\to \infty}\frac{H_n(\mu)}{\log 2^n}.
\]
The \textit{lower entropy dimension} $\underline{\dim}_e \mu$ is defined similarly by taking the lower limit. If the upper entropy dimension and lower entropy dimension are equal, then we call this common value the \textit{entropy dimension} of $\mu$.

Let $\mu^x, \mu^y$ be the projection of  $\mu_{A,\mathcal{D}}$ onto the $x$ and $y$-axes, respectively, i.e., $\mu^x=\mu\circ\pi_x^{-1}$ and $\mu^y=\mu\circ\pi_y^{-1}$.
\begin{prop}\label{propeed}
 We have
\[\dim_e\mu_{A,\mathcal{D}}=\frac{\dim_e\mu^x\cdot\log\frac{3q_2}{3q_1}+\log 3}{\log (3q_2)}\in\left(\frac{\log 3}{\log (3q_2)},\frac{\log 3}{\log (3q_1)}\right),
 \]
where
 \[
 \dim_e\mu^x=\frac{\frac{2}{3}\log \frac{2}{3}+\frac{1}{3}\log\frac{1}{3}}{-\log (3q_1)}.
 \]
\end{prop}

To prove it, we need the following two results. The first one was proved in \cite[Theorem 4]{FW05} and the second one was proved by Hutchinson in \cite{Hut}, see also \cite{LW}.

\begin{lemma}\label{leentro}
Let $\Phi=\{\tau_b(x)=A^{-1}(x+b):b\in \mathcal{B}\}$ be a self-affine  IFS satisfying the ROSC. Then
\begin{equation*}
	\dim_e(\mu_{A,\mathcal{B}})=
	\begin{cases}
		\frac{\dim_e\mu^x\cdot\log\frac{m}{n}+\log\#\mathcal{B}}{\log m},& \text{if $n<m$;}\\
		\frac{\dim_e\mu^y\cdot\log\frac{n}{m}+\log\#\mathcal{B}}{\log n},& \text{if $n\geq m$.}
	\end{cases}
\end{equation*}
Here, we say that $\Phi=\{\tau_b(x)=A^{-1}(x+b):b\in \mathcal{B}\}$ satisfies the {\it rectangular open set condition (ROSC)}, if there exists an open rectangle $T=(0,R_1)\times (0,R_2)+v$ such that $\{\tau_b(T)\}_{b\in\mathcal{B}}$ are disjoint subsets of $T$.
\end{lemma}

\begin{lemma}\label{leentrow}
Let $\mu$ be the self-similar measure associated with an IFS $\{r_ix+t_i\}_{1\leq i\leq m}$ and probability weight $\{p_i\}_{1\leq i\leq m}$. Under the OSC for the IFS, we have that
\[
\hdim \mu=\frac{\sum_{i=1}^mp_i\log p_i}{\sum_{i=1}^mp_i\log r_i},
\]
where $\hdim\mu$ denotes the Hausdorff dimension of $\mu$.
\end{lemma}

\proof[Proof of Proposition  \ref{propeed}] It is easy to check that  $\mu^x$ is the self-similar measure associated with the IFS $\Phi=\{\frac{1}{3q_1}x,\frac{1}{3q_1}x,\frac{1}{3q_1}(x+1)\}$ and probability weight $(\frac{1}{3},\frac{1}{3},\frac{1}{3})$. The system is equivalent to the IFS $\Phi'=\{\frac{1}{3q_1}x,\frac{1}{3q_1}(x+1)\}$ and the probability weight $\left\{\frac{2}{3},\frac{1}{3}\right\}$. Clearly, $\Phi'$ satisfies the OSC and therefore the corresponding measure $\mu^x$ is exact dimensional, see \cite{FH09}. This, combined with Lemma \ref{leentrow} and Theorem 1.1 in \cite{FLR02}, implies that
\begin{equation}\label{eqselen}
\dim_e\mu^x=\hdim\mu^x=\frac{\frac{2}{3}\log \frac{2}{3}+\frac{1}{3}\log \frac{1}{3}}{\frac{2}{3}\log \frac{1}{3q_1}+\frac{1}{3}\log \frac{1}{3q_1}}=\frac{\frac{2}{3}\log \frac{2}{3}+\frac{1}{3}\log \frac{1}{3}}{\log \frac{1}{3q_1}}.
\end{equation}

On the other hand, it is easy to see that $\{\tau_d(x)=A^{-1}(x+d):d\in \mathcal{D}\}$ is a self-affine IFS satisfying the ROSC. Then by Lemma \ref{leentro},
\[
\dim_e\mu_{A,\mathcal{D}}=\frac{-\dim_e \mu^x\cdot\log \frac{3q_2}{3q_1}+\log \frac{1}{3}}{\log \frac{1}{3q_2}}.
\]
Substituting \eqref{eqselen} into the above equation, we obtain $$\dim_e\mu_{A,\mathcal{D}}\in\left(\frac{\log 3}{\log (3q_2)},\frac{\log 3}{\log (3q_1)}\right).$$
\qed

\section{Intermediate value property}

In this section, we will prove Theorem \ref{thexu1}. Recall that in the regular case the set $\tau^*(\Theta_3^\tau)$ can be expressed  as a sequence as follows. For $I=i_1i_2\cdots i_n\in\Theta_3^n\setminus\Theta_3^{n-1}0$, there exists a unique integer
  $$k_I=i_1+i_2 3+\cdots+i_n 3^{n-1}$$
  and on the other hand
\[
\mathbb{Z}\setminus\{0\}=\bigcup_{n=1}^\infty\bigcup_{I\in\Theta_3^n\setminus\Theta_3^{n-1}0} k_I.
\]
Define $\lambda_0=\textbf{0}$ and $\lambda_k=\tau^*(i_1i_2\cdots i_n0^\infty)$ if
$$k=i_1+i_2 3+\cdots+i_n 3^{n-1}\quad \text{with $i_j\in\Theta_3$ and $i_n\neq0$}.$$
Then $\tau^*(\Theta_3^\tau)=\{\lambda_k\}_{k\in\mathbb{Z}}$.

Let $M=\{m_k\}_{k=-\infty}^\infty$ be a sequence of non-negative integers. We define a mapping $\tau$ as follows: $\tau(0^i)=\textbf{0}$ for $i\ge 1$; for $I=i_1i_2\cdots i_n\in\Theta_3^n$ with $I\neq 0^n$, $\tau(I)=i_n\begin{pmatrix}
		q_1 \\
		-q_2 \\
	\end{pmatrix}$, and
\[
	\tau(I0^l)=
	\begin{cases}
		\textbf{0},& \text{if $l\not=m_k;$}\\
		\begin{pmatrix}
		\frac{q_1}{4} \\
		\frac{-q_2}{4} \\
	\end{pmatrix},& \text{if $l=m_k.$}
	\end{cases}
\]
Define $\lambda_0=\textbf{0}$ and
\[
\lambda_k=\sum_{j=1}^{n} A^{j-1}\tau(i_1\cdots i_j)+A^{n+m_k-1}\begin{pmatrix}
		\frac{q_1}{4} \\
		\frac{-q_2}{4} \\
	\end{pmatrix}\delta_{m_k},
\]
where
	\begin{equation*}
	\delta_{m_k}=
	\begin{cases}
		0,& \text{if $ ~m_k=0;$}\\
		1,& \text{if $ ~m_k\geq 1.$}
	\end{cases}
\end{equation*}
Then, it is easy to see that  $\tau$ is a maximal tree mapping and $\ell_n:=\#\{k:\tau(I0^l)\neq \textbf{0},l\geq1\}\leq 1$ for all $n\geq1$. By Theorem \ref{ppd112}, we have that $\Lambda(\{m_k\}):=\{\lambda_k\}_{k=-\infty}^\infty$ is a spectrum of $\mu_{A,\mathcal{D}}.$

When $m_n=0$ for all $n\geq1$, it is easy to see that
\begin{equation}\label{maxL}
\Lambda_{\max}:=\Lambda(\{m_k\})=\bigcup_{k=1}^\infty \left\{\sum_{i=1}^k A^{i-1}l_i:~\text{$l_i\in \mathcal{L}$ for $1\le i\le k$}\right\}.
\end{equation}
It follows from Theorem 1.4 in \cite{LWu322} that
\[
\dim_{Be} \Lambda_{\max}=\frac{\log 3}{\log 3q_2}.
\]

Before proceeding with the proof, we introduce an equivalent definition of the Beurling dimension of a set and the dimensional formulas for some special discrete sets, which will be used in the latter.
\begin{lemma}[\cite{CKS08}]\label{eqbe}
Let $\Lambda\subseteq \mathbb{R}^d$ be a countable set. Then
\begin{equation*}
\dim_{Be}(\Lambda)=\limsup\limits_{h\to \infty}\sup_{x\in \mathbb{R}^d}\frac{\log \#(\Lambda \cap B(x,h))}{\log h}.
\end{equation*}
\end{lemma}

\begin{lemma}\label{leonele}
Let $b\geq 2$ be an integer. Let $D\subseteq\mathbb{Z}$ be a  digit set with $D\subseteq\{-\big[\frac{b}{2}\big],-\big[\frac{b}{2}\big]+1,\ldots,b-1-\big[\frac{b}{2}\big]\}$. Define
\[
\Lambda(b,D):=\bigcup_{k=1}^\infty \left\{\sum_{i=1}^k b^{i-1}d_i:~\text{$d_i\in D_i$ for $1\le i\le k$}\right\},
\]
where $D_i=D$ or $D_i=\{0\}$ for any $i\geq 1$.
Then
\[
\dim_{Be}\Lambda(b,D)=\limsup\limits_{n\to \infty}\frac{\#\{i: D_i=D, 1\leq i\leq n\}}{n}\cdot\frac{\log\#D}{\log b}.
\]
\end{lemma}
\begin{proof}
Write $\overline{d}=\limsup\limits_{n\to \infty}\frac{\#\{i: D_i=D, 1\leq i\leq n\}}{n}$. Fix $\epsilon>0$. Note that for all $x\in\mathbb{R}$ and sufficiently large $h$, there exists sufficiently large $n\in\N$ such that $b^n\leq h\leq b^{n+1}$, and then
\[
\#(\Lambda(b,D)\cap B(x,h))\leq 2\cdot(\#D)^{(\overline{d}+\epsilon)\cdot(n+1)}.
\]
Hence,
\[
\frac{\log \#(\Lambda(b,D)\cap B(x,h))}{\log h}\leq\frac{\log(2\cdot(\#D)^{(\overline{d}+\epsilon)\cdot (n+1)})}{\log(b^n)}.
\]
 By Lemma \ref{eqbe}, we have that $\dim_{Be}\Lambda(b,D)\leq \overline{d}\cdot \frac{\log \#D}{\log b}$.\\
On the other hand, let
\[
\Lambda_k=\bigcup_{l=1}^k \left\{\sum_{i=1}^l b^{i-1}d_i:~\text{$d_i\in D_i$ for $1\le i\le l$}\right\},~~k\geq 1.
\]
Then $\Lambda(b, D)=\cup_{k=1}^\infty \Lambda_k$. If we write $c=\frac{1}{b-1}\max_{d\in D}|d|$ and $c_k=\frac{\#\{i: D_i=D, 1\leq i\leq k\}}{k}$, then, for sufficiently large $k$, $\Lambda_k\subseteq B(0,cb^k)$ and hence
\[
\begin{split}
\dim_{Be}(\Lambda(b,D))&=\limsup\limits_{h\to \infty}\sup_{x\in \mathbb{R}}\frac{\log \#(\Lambda(b, D)\cap B(x,h))}{\log h}\\
                  &\geq \limsup\limits_{k\to \infty}\frac{\log \#(\Lambda_k \cap B(0,cb^k))}{\log cb^k}\\
                  &\geq \limsup\limits_{k\to \infty}\frac{\log \#(\Lambda_k \cap B(0,cb^k))}{\log cb^k}\\
                 &\geq \limsup\limits_{k\to \infty}\frac{\log \#(\Lambda_k)}{\log cb^k}\\
                  &= \limsup\limits_{k\to \infty}\frac{\log (\#D)^{k\cdot c_k}}{\log cb^k}\\
		          &=\overline{d}\cdot \frac{\log \#D}{\log b},
\end{split}
\]
where the second equality holds because $D\subseteq\{-\big[\frac{b}{2}\big],-\big[\frac{b}{2}\big]+1,\ldots,b-1-\big[\frac{b}{2}\big]\}$. Hence,  $\dim_{Be}(\Lambda(b, D))=\overline{d}\cdot\frac{\log \#D}{\log b}$.
\end{proof}

With the help of Lemma \ref{leonele}, we obtain the following similar dimensional formula on $\mathbb{R}^2.$

\begin{lemma}\label{leseag}
Let $R=\begin{pmatrix}
		a & 0 \\
		0 & b \\
	\end{pmatrix}$ with $1<a\leq b$ and $a,b\in\mathbb{N}$. Let $B\subseteq \mathbb{Z}^2$ be a finite digit set with $\pi_y(B)\subseteq\{-\big[\frac{b}{2}\big],-\big[\frac{b}{2}\big]+1,\ldots,b-1-\big[\frac{b}{2}\big]\}$. Define
\[
\Lambda(R,\{B_i\}):=\bigcup_{k=1}^\infty \left\{\sum_{i=1}^k R^{i-1}b_i:~\text{$b_i\in B_i$ for $1\le i\le k$}\right\},
\]
where $B_i=B$ or $B_i=\{\textbf{0}\}$. If $\Lambda(R,\{B_i\})$ has at most one point on every vertical line, then
\[
\dim_{Be}\Lambda(R,\{B_i\})=\limsup\limits_{n\to \infty}\frac{\#\{i: B_i=B, 1\leq i\leq n\}}{n}\cdot\frac{\log \#B}{\log b}.
\]
\end{lemma}
\begin{proof}
Since $\Lambda(R,\{B_i\})$ has at most one point on every vertical line, we have
\[
\#(\Lambda(R,\{B_i\})\cap B(x,h))\le \#(\pi_y(\Lambda(R,\{B_i\}))\cap B(x^{(2)},h))
\]
for any $h>0$ and $x=\begin{pmatrix}
		x^{(1)}\\
		x^{(2)}\\
	\end{pmatrix}\in\mathbb{R}^2$. Similarly as in the proof of Theorem \ref{thdddll}, we have that
\[
\dim_{Be}\Lambda(R,\{B_i\})\leq \dim_{Be}\pi_y(\Lambda(R,\{B_i\})).
\]
Note that
\[
\pi_y(\Lambda(R,\{B_i\}))=\bigcup_{k=1}^\infty \left\{\sum_{i=1}^k b^{i-1}d_i:~\text{$d_i\in \pi_y(B)$ for $1\le i\le k$}\right\}.
\]
By the condition that $\Lambda(R,\{B_i\})$ has at most one point on every vertical line again, we have that $\#\pi_y(B)=\#B$. Then by Lemma \ref{leonele}, we have
\[
\dim\pi_y(\Lambda(R,\{B_i\}))=\limsup\limits_{n\to \infty}\frac{\#\{i: B_i=B, 1\leq i\leq n\}}{n}\cdot\frac{\log \#B}{\log b}.
\]
On the other hand, similarly as in Lemma \ref{leonele}, write $\overline{b}=\limsup\limits_{n\to \infty}e_n$, $e_n=\frac{\#\{i: B_i=B, 1\leq i\leq n\}}{n}$ and
\[
\Lambda_k=\bigcup_{l=1}^k \left\{\sum_{i=1}^l R^{i-1}b_i:~\text{$b_i\in B_i$ for $1\le i\le l$}\right\}, ~~k\geq 1.
\]
Then for sufficiently large $k$, $\Lambda_k\subseteq B(\textbf{0},cb^k)$ for some $c>0$ and thus
\[
\begin{split}
\dim_{Be}(\Lambda(R,\{B_i\}))&=\limsup\limits_{h\to \infty}\sup_{x\in \mathbb{R}^2}\frac{\log \#(\Lambda(R,\{B_i\})\cap B(x,h))}{\log h}\\
                             &\geq \limsup\limits_{k\to \infty}\frac{\log \#(\Lambda(R,\{B_i\})\cap B(\textbf{0},cb^k))}{\log cb^k}\\
		               &\geq \limsup\limits_{k\to \infty}\frac{\log (\#B)^{k\cdot e_k}}{\log cb^k}\\
                       &=\overline{b}\cdot \frac{\log \#B}{\log b}.
\end{split}
\]
Hence, $\dim_{Be}(\Lambda(R,\{B_i\}))=\overline{b}\cdot\frac{\log \#B}{\log b}$.
\end{proof}

Now we proceed with the proof of Theorem \ref{thexu1}.
\begin{prop}\label{indt}
For any $t\in\Big[0,\frac{\log 3}{\log 3q_2}\Big]$, there exists a subset $F_t\subseteq\Lambda_{\max}$ (recall that $\Lambda_{\max}$ is defined by \eqref{maxL}) such that
\[
\dim_{Be} F_t=t.
\]
\end{prop}

\begin{proof}

 Choose a sequence $\{D_i\}$ such that $D_i=\mathcal{L}$ or $D_i=\{\textbf{0}\}$ and
\[
\limsup\limits_{n\to \infty}\frac{\#\{i: D_i=\mathcal{L}, 1\leq i\leq n\}}{n}=t\cdot \frac{\log (3q_1)}{\log 3}.
\]
Then the corresponding set $\Lambda(A,\{D_i\})\subseteq \Lambda_{\max}$  satisfies that $$\dim_{Be}(\Lambda(R,\{D_i\}))=t$$ due to Lemma \ref{leseag}.
\end{proof}

Write $F_t=\{\lambda_k\}_{k\in\Gamma_t}$. Choose $m_k=k^2$ for all $k\notin \Gamma_t$ and $m_k=0$ for otherwise. For this sequence $\{m_k\}$, define $\Lambda_t:=\Lambda(\{m_k\})$. Then we can write
\[\begin{split}\Lambda_t&=\{\lambda_k\}_{k\in \Gamma_t}\cup\{\lambda_k\}_{k\notin \Gamma_t}\\
	&=:F_t\cup\Lambda_t'.
\end{split}	\]

We have known that $\Lambda_t$ is a spectrum of $\mu_{A,\mathcal{D}}$ by the previous arguments. So, we next need to prove that $\dim_{Be}\Lambda_t=t$. To do it, we need some basic properties of Beurling dimension.
\begin{lemma}[\cite{CKS08}]\label{inc}
	Let $\Lambda_1,\Lambda_2\subseteq \mathbb{R}^2$ be two countable sets. Then the following statements hold:
\begin{enumerate}
\item [\textup{(i)}] Monotonicity: If $\Lambda_1\subseteq \Lambda_2$, then
\[
\dim_{Be}\Lambda_1\leq \dim_{Be}\Lambda_2;
\]
\item [\textup{(ii)}] Finite stability:
\[
\dim_{Be}(\Lambda_1\cup \Lambda_2)=\max(\dim_{Be}\Lambda_1,\dim_{Be}\Lambda_2).
\]
\end{enumerate}
	
\end{lemma}
By Proposition \ref{indt} and the finite stability of Beurling dimension, we only need to show that $\dim_{Be}\Lambda_t'=0.$ A countable set $\Lambda=\{a_n\}_{n=-\infty}^{\infty}\subseteq \mathbb{R}^2$ is called {\it $b$-lacunary}, if $a_0 = 0$, $|a_1|\geq b$ and for all $n \geq 1$, $|a_{n+1}| \geq b|a_n|$; if $|a_{-1}|\geq b$ and for all $n \geq 1$, $|a_{-n-1}| \geq b|a_{-n}|$.

\begin{lemma}[\cite{AL20}]\label{lac}
	Let $b>1$. If $\Lambda$ is a $b$-lacunary set, then $\dim_{Be}\Lambda=0$.
\end{lemma}

\begin{prop}\label{zeroset}
	$\dim_{Be}\Lambda_t'=0$.
\end{prop}

\begin{proof}
Choose $m_k=k^2$ for all $k\in\mathbb{Z}$. Consider the set $\Lambda(\{|k|\})=\{\lambda_k\}_{k=-\infty}^{\infty}$. 	
Recall that for $k\neq0$, there exists $n:=n(k)$ such that $k=i_1+i_23+\cdots+i_n3^{n-1}$ with $i_j\in\Theta_3$ and $i_n\neq0$, and
\begin{equation*}
\begin{split}\lambda_k&=
\sum_{j=1}^{n(k)} A^{j-1}i_j\begin{pmatrix}
		q_1 \\
		-q_2 \\
	\end{pmatrix}+A^{n(k)+k^2-1}\begin{pmatrix}
		\frac{q_1}{4} \\
		\frac{-q_2}{4} \\
	\end{pmatrix}\\
&=\begin{pmatrix}
		 \sum_{j=1}^{n(k)}(3q_1)^{j-1}i_jq_1+(3q_1)^{n(k)+k^2-1}\frac{q_1}{4}\\
		 \sum_{j=1}^{n(k)}(3q_2)^{j-1}i_j(-q_2)+(3q_2)^{n(k)+k^2-1}\frac{-q_2}{4}\\
	\end{pmatrix}.
\end{split}
\end{equation*}
Without loss of generality, we only need to consider the case that $k>0$. Write $\lambda_k=\begin{pmatrix}
		\lambda_k^{(1)} \\
		\lambda_k^{(2)} \\
	\end{pmatrix}$. Then, $|\lambda_k^{(1)}|\geq \frac{1}{2}(3q_1)^{n(k)+k^2-1}\frac{q_1}{4}$ and $|\lambda_k^{(2)}|\geq \frac{1}{2}(3q_2)^{n(k)+k^2-1}\frac{q_2}{4}$.
Moreover, $|\lambda_k^{(1)}|\leq 2(3q_1)^{n(k)+k^2-1}\frac{q_1}{4}$ and $|\lambda_k^{(2)}|\leq 2(3q_2)^{n(k)+k^2-1}\frac{q_2}{4}$. Hence,
\[\frac{|\lambda_{k+1}^{(1)}|}{|\lambda_{k}^{(1)}|}\geq \frac{\frac{1}{2}(3q_1)^{n(k+1)+(k+1)^2-1}\frac{q_1}{4}}{2(3q_1)^{n(k)+k^2-1}\frac{q_1}{4}}\geq\frac{(3q_1)^{2k}}{4},\]
where the last inequality holds because $\{n(|k|)\}$ is not decreasing. Similarly,
\[\frac{|\lambda_{k+1}^{(2)}|}{|\lambda_{k}^{(2)}|}\geq \frac{\frac{1}{2}(3q_2)^{n(k+1)+(k+1)^2-1}\frac{q_2}{4}}{2(3q_2)^{n(k)+k^2-1}\frac{q_2}{4}}\geq\frac{(3q_2)^{2k}}{4}.\]
By the fact that for any four positive numbers $x_1,x_2,y_1,y_2$, if $\frac{x_1}{y_1}\geq a>1$ and $\frac{x_2}{y_2}\geq b>1$, then $\frac{\sqrt{x_1^2+x_2^2}}{\sqrt{y_1^2+y_2^2}}\geq \min\{a,b\}$, we have
\[\frac{|\lambda_{k+1}|}{|\lambda_{k}|}\geq\frac{(3q_1)^{2k}}{4}\geq\frac{(3q_1)^{2}}{4}>1.\]
It follows from Lemma \ref{lac} that $\dim_{Be}(\Lambda(\{|k|\}))=0$. Observe that $\Lambda_t'\subseteq \Lambda(\{|k|\})$. Therefore,  $\dim_{Be}\Lambda_t'=0$ due to Lemma \ref{inc}.
\end{proof}

 Finally, we will prove the second assertion of Theorem \ref{thexu1}. For any $t\in\left[0,\frac{\log3}{\log 3q_2}\right]$.  For any $n\notin\Gamma_t$, we let $m_n=n^2$ or $n^2+1$. Then it is easy to see that $m_n<m_{n+1},n\in\N$. By choosing $m_n,n\in \Gamma_t^c$ randomly from the above two choices (note that the sets $\Gamma_t,\Gamma_t^c$ are all countable), we obtain that
 the level set
\[
L_t:=\{\Lambda:  \text{$\Lambda$ is a spectrum of $\mu_{A,\mathcal{D}}$ and $\dim \Lambda =t$} \}
\]
has the cardinality of the continuum.
The proof of Theorem \ref{thexu1} is completed.

% ------------------------------------------------------------------------

\subsection*{Acknowledgements}
The project was supported by the National Natural Science Foundations of China (12171107, 12271534, 11971109), Guangdong NSF (2022A1515011844) and the Foundation of Guangzhou University (202201020207, RQ2020070).

\end{document}